\documentclass{amsart}
\usepackage{amsfonts,amssymb}
\usepackage{enumerate,graphicx}
\usepackage{caption}
\usepackage{subcaption}
\usepackage[square,numbers]{natbib}
\usepackage{mathrsfs,bbm}
\usepackage{tikz,tikzscale,tikz-cd}
\usepackage{multirow,xparse,fp}
\usepackage{environ}
\usepackage{amstext}
\usepackage{hyperref}
\usepackage{diagbox}
 \usepackage[pagewise]{lineno}
 \usepackage{qtree}
\allowdisplaybreaks[4]

\newtheorem{theorem}{Theorem}[section]

\newtheorem{proposition}[theorem]{Proposition}

\newtheorem{lemma}[theorem]{Lemma}
\theoremstyle{definition}

\newtheorem{definition}[theorem]{Definition}
\newtheorem{example}[theorem]{Example}

\newtheorem{remark}[theorem]{Remark}

\numberwithin{equation}{section}
\begin{document}

\title{On spatial entropy and periodic entropies of Two-dimensional Shifts of Finite Type}

\author[Wen-Guei Hu]{Wen-Guei Hu}
\address[Wen-Guei Hu]{College of Mathematics, Sichuan University, Chengdu, 610064, China}
\email{wghu@scu.edu.cn}

\author[Guan-Yu Lai]{Guan-Yu Lai}
\address[Guan-Yu Lai]{Department of Applied Mathematics, National Yang Ming Chiao Tung University, Hsinchu 30010, Taiwan, ROC.}
\email{guanyu.am04g@g2.nctu.edu.tw}

\author[Song-Sun Lin]{Song-Sun Lin}
\address[Song-Sun Lin]{Department of Applied Mathematics, National Yang Ming Chiao Tung University, Hsinchu 30010, Taiwan, ROC.}
\email{sslin@nctu.edu.tw}

\thanks{ Hu is partially supported by the National Natural Science Foundation of China (Grant 11601355). Lai is partially supported by the Ministry of Science and Technology, ROC (Contract MOST 111-2811-M-004-002-MY2). Lin is partially supported by the Ministry of Science and Technology, ROC (Contract MOST 106-2115-M-009-004-, 107-2115-M-009-006-, 108-2115-M-009-005- and 109-2115-M-009-004-).}

\date{}

\baselineskip=1.2\baselineskip

\begin{abstract}
Topological entropy or spatial entropy is a way to measure the complexity of shift spaces. This study investigates the relationships between the spatial entropy and the various periodic entropies which are computed by skew-coordinated systems $\gamma\in GL_2(\mathbb{Z})$ on two dimensional shifts of finite type. It is known that there are some aperiodic two dimensional shifts of finite types with positive spatial entropy without any periodic patterns. Hence, the spatial entropy is strictly greater than periodic entropy which is zero. On the other hand, when the shift spaces have some mixing properties then these two entropies are equal. In this paper, we show that some periodic mixing properties imply all of these entropies are equal. Indeed, for two dimensional shift of finite type $\Sigma(\mathcal{B})$ which is generated by an admissible local patterns $\mathcal{B}\subseteq \{0,1,...,r\}^{\mathbb{Z}_{m\times m}}, r\geq 1$ and $m\geq 2$. The horizontal periodic transition matrix $\mathbf{T}_m(\mathcal{B})$ is a 0-1 matrix which stores horizontal periodic patterns with period $m\geq 1$ and height 2 with the maximum eigenvalue $\rho(\mathbf{T}_m(\mathcal{B}))$. Sequence $\{\mathbf{T}_m(\mathcal{B})\}_{m=1}^\infty$ is called uniformly dominated by $\{\rho(\mathbf{T}_m(\mathcal{B}))\}_{m=1}^\infty$ if $|\mathbb{T}_m^k(\mathcal{B})| \leq c(m,k)\rho(\mathbf{T}_m(\mathcal{B}))^k$ with $\frac{1}{mk}\log c(m,k)$ tends to zero as $(m,k)$ tends to infinte. Uniform dominanting follows if the associated graphs of transition matrices $\{\mathbf{T}_m(\mathcal{B})\}_{m=1}^\infty$ have uniformly bounded diameters or the shift space $\Sigma(\mathcal{B})$ has periodic-block gluing properties. The explicit formulas for the transformations between the Hermite normal forms of the finite-index subgroups of $\mathbb{Z}^2$ in different $\gamma$ coordinated system are obtained. Therefore, it can be proved that uniformly dominant conditions (0.1) and (0.2) implies the spatial entropy of the shift space $\Sigma(\mathcal{B})$ is equal to the periodic entropy which is computed by $\gamma$-coordinated system for all $\gamma \in GL_2(\mathbb{Z}) $. In these case, the spatial entropy can also be computed as $\limsup_{q\to\infty}\log \rho({\bf T}_{\gamma_q,1})$ where $\gamma_q=\left[\begin{matrix}
    1&q\\
    0&1
\end{matrix}\right]$.
\end{abstract}
\maketitle

\section{Introduction}
Topological entropy or spatial entropy is a way to measure the complexity of shift spaces $X$ which has been studied extensively for many years, see \cite{BL2005 Patterns generation and transition matrices in multi-dimensional lattice models,BHL2019 Pattern generation problems arising in multiplicative integer systems,BLL2007 Patterns generation and spatial entropy in two dimensional lattice models,BLL2009 Mixing property and entropy conjugacy of Z2 subshift of finite type: a case study,Baxter1971 Eight-vertex model in lattice statistics,Baxter1982 Exactly Solved Models in Statistical Mechanics,Boyle2010 Multidimensional sofic shifts,Ceccherini-Silberstein2012 The myhill property,Ceccherini-Silberstein2012 On the density,CHLLY2022,Chow1995 Dynamics,Haggstrom1995 A subshift,Haggstrom1995 On the relation,Haggstrom1996 On phase transitions,Hochman2010 A characterization,Johnson2005 Factoring,Joyce1988 On the Hard-Hexagon,Juang2000 Cellular Neural,Lieb1967 Exact solution,Lieb1967 Exact solution of the F model1967,Lieb1967 Exact solution of the two-dimensional,Lieb1967 The residual entropy,Lieb1970 Ice ferro- and,Markley1981 Matrix subshifts,Meester2001 Higher-dimensional,Quas2000 Subshifts of multidimensional}. The periodic entropy only counts the subset $X_p$ of periodic patterns of $X$. For multi-dimensional symbolic dynamical systems $X$ (or $\mathbb{Z}^d,d\geq2$ actions). Apart from the rectangular-periodicity, the skew-periodicity can also be considered in each coordinated systems on $\mathbb{Z}^d$ \cite{BHLL2013 Zeta functions for two-dimensional shifts of finite type}. This study investigates the relationships between the spatial entropy and various periodic entropies in skew-coordinated systems.  

For simplicity, this study presents only the case of two-dimensional shifts of finite type. Let $\mathbb{Z}_{m\times m}$ be the $m\times m$ square lattice in $\mathbb{Z}^2$ and $\mathcal{S}$ be the finite set of symbols $\mathcal{S}=\left\{ 0,1,...,r-1\right\}$, $r\geq 2$. Denote by $\mathcal{S}^{\mathbb{Z}_{m\times m}}$ be the set of all local patterns on $\mathbb{Z}_{m\times m}$. A given subset $\mathcal{B}\subseteq \mathcal{S}^{\mathbb{Z}_{m\times m}}$ is called a basic set of admissible local patterns. $\Sigma \left(\mathcal{B} \right)$ is the set of all admissible (global) patterns generated by $\mathcal{B}$ on $\mathbb{Z}^2$. $\Sigma_{p}\left(\mathcal{B}\right)$ is the subset of all rectangular periodic patterns of $\Sigma \left( \mathcal{B}\right)$. The spatial entropy of $\Sigma \left(\mathcal{B} \right)$ is defined by
\begin{equation}
h(\mathcal{B})=\lim_{(n,k)\rightarrow \infty}\frac{1}{nk} \log \left|\Sigma_{n\times k}\left(\mathcal{B}\right) \right|,
\end{equation}
where $\left|\Sigma_{n\times k}\left(\mathcal{B}\right) \right|$ is the cardinal number of $\Sigma_{n\times k}\left(\mathcal{B}\right)$ which is the set of all admissible patterns generated by $\mathcal{B}$ on sublattice $\mathbb{Z}_{n\times k}$.

The skew-coordinated system can be represented by unimodular transformation $\gamma \in GL_2\left( \mathbb{Z}\right)$. The modular group is defined by 
\begin{equation*}
GL_2\left(\mathbb{Z}\right)=\left\{\left[\begin{matrix}
a&b\\
c&d
\end{matrix} \right]: a,b,c,d\in \mathbb{Z} \mbox{ and } \left| ad-bc\right|=1   \right\}.
\end{equation*}
$\gamma_0=\left[\begin{matrix}
1&0\\0&1
\end{matrix}\right]\in GL_2\left(\mathbb{Z}\right)$ is the standard rectangular system. The finite-index subgroup of $\mathbb{Z}^2$ can be parameterized in Hermite normal form \cite{Lind1998 A zeta function for,Mac Duffie1956 The theory of matrices} as 
\begin{equation*}
\mathcal{L}_2=\left\{  \left[  \begin{matrix}
n&\ell \\
0&k
\end{matrix}  \right]\mathbb{Z}^2 : n\geq 1, k\geq 1\mbox{ and }0\leq \ell \leq n-1
\right\}.
\end{equation*}
In the $\gamma$-system,
\begin{equation*}
\mathcal{L}_2=\left\{  \left[  \begin{matrix}
n&\ell \\
0&k
\end{matrix}  \right]_{\gamma}\mathbb{Z}^2 : n\geq 1, k\geq 1\mbox{ and }0\leq \ell \leq n-1
\right\}.
\end{equation*}
The spatial entropy $h_{\gamma}\left(\mathcal{B}\right)$ computed in $\gamma$ system is defined by 
\begin{equation}
h_{\gamma}(\mathcal{B})=\limsup_{(n,k)\rightarrow \infty}\frac{1}{nk} \log \left|\Sigma_{\gamma;n\times k}\left(\mathcal{B}\right) \right|.
\end{equation}
The periodic entropy $h_{\gamma,p}\left(\mathcal{B}\right)$ of basic set of admissible patterns $\mathcal{B}$ is defined by
\begin{equation}
h_{\gamma,p}\left(\mathcal{B}\right)=\limsup_{(n,k)\rightarrow\infty} \sup_{0\leq \ell \leq n-1} \frac{1}{nk} \log \Gamma_{\mathcal{B}}\left(  \left[  \begin{matrix}
n&\ell \\
0&k
\end{matrix}  \right]_{\gamma} \right).
\end{equation}
When $\gamma=\gamma_0$, $h_{\gamma_0,p}\left(\mathcal{B}\right)$ is also denoted by $h_{p}\left(\mathcal{B}\right)$ for simplicity. 

According to the result of Hu and Lin \cite{HL2016 On spatial entropy}, the spatial entropy $h_{\gamma}\left(\mathcal{B}\right)$ computed by any skew-system $\gamma$ is equal to the standard coordinated system $h\left( \mathcal{B}\right)$, i.e.,
\begin{equation}
h_{\gamma}\left(\mathcal{B}\right)=h\left(\mathcal{B}\right)
\end{equation}
for all $\gamma\in GL_2\left(\mathbb{Z}\right)$. 

We begin with the study by considering when $h\left(\mathcal{B}\right)=h_{p}\left(\mathcal{B}\right)$ in the standard coordinated system $\gamma_0=\left[\begin{matrix}
1&0\\0&1
\end{matrix}\right]$.

It is known when $r\geq 5$, there are some basic sets $\mathcal{B}^*\subseteq \mathcal{S}_r^{\mathbb{Z}_{2\times 2}}$ with $h\left(\mathcal{B}^*\right)>0$ and $\Sigma_p\left(\mathcal{B}^*\right)=\emptyset$. Therefore
\begin{equation}
h\left(\mathcal{B}^*\right)>h_p\left(\mathcal{B}^*\right),
\end{equation}
see \cite{Culik II1996 An aperiodic set,Kari1996 A small aperiodic} and related works \cite{Berger1966 The undecidability,CHL2020 Nonemptiness problems,Penrose1974 The role of aesthetics,Wang1961 Proving theorems}. On the other hand, when $\Sigma\left(\mathcal{B}\right)$ has some mixing property then
\begin{equation}\label{1-6}
h\left(\mathcal{B}\right)=h_p\left(\mathcal{B}\right),
\end{equation}
see \cite{Boyle2010 Multidimensional sofic shifts,Ward1994 Automorphisms} and related works \cite{Kass2013 A note on the,Lightwood2003 Morphisms from,Lightwood2004 Morphisms from non-periodic Z2 subshifts II,Markley1979 Maximal measures,Markley1981 Matrix subshifts,Meester2001 Higher-dimensional,Quas2000 Subshifts of multidimensional,Ruelle1973 Statistical mechanics,Ruelle1978 Thermodynamic Formalism,Ruelle1992 Thermodynamic formalism for maps satisfying}. In this paper, we introduce a kind of periodic-mixing property to ensure (\ref{1-6}) hold. In the following, we briefly introduce the relevant notations and the notation of the periodic-mixing property.

Based on our earlier studies \cite{BHLL2013 Zeta functions for two-dimensional shifts of finite type,BLL2007 Patterns generation and spatial entropy in two dimensional lattice models}, we introduce the horizontal-periodic transition matrix ${\bf T}_m\left(\mathcal{B}\right)$ for all $m\geq 1$. ${\bf T}_m\left(\mathcal{B}\right)$ stores all horizontal periodic patterns which are generated by $\mathcal{B}$ on $\mathbb{Z}_{m\times 2}, m\geq 1$. Denoted by $\rho\left( {\bf T}_m\left(\mathcal{B}\right)\right)$ to be the maximum eigenvalue of matrix ${\bf T}_m\left(\mathcal{B}\right)$. The sequence of horizontal periodic transition matrices $\left\{ {\bf T}_m\left(\mathcal{B}\right) \right\}_{m=1}^\infty$ is called uniformly dominated by $\left\{ \rho\left({\bf T}_m\left(\mathcal{B}\right)\right) \right\}_{m=1}^\infty$ if 
\begin{equation}\label{1-7}
\left|{\bf T}_m^k\left(\mathcal{B}\right) \right|\leq c(m,k)\rho\left({\bf T}_m\left(\mathcal{B}\right)\right)^k
\end{equation}
with
\begin{equation}\label{1-8}
\lim_{(m,k)\to\infty}\frac{1}{mk}\log c(m,k)=0.
\end{equation}

Our first result is the following theorem.
\begin{theorem}\label{thm0-1}
	If (\ref{1-7}) and (\ref{1-8}) hold, then 
	\begin{equation*}
	h\left(\mathcal{B}\right)=h_p\left(\mathcal{B}\right)=h_*\left(\mathcal{B}\right),
	\end{equation*}
where 	
\begin{equation}
h_*\left(\mathcal{B}\right)=\limsup_{m\to \infty}\frac{1}{m}\log \rho\left({\bf T}_m\left(\mathcal{B}\right)\right).
\end{equation}
\end{theorem}

Conditions (\ref{1-7}) and (\ref{1-8}) can be obtained by checking the following connectedness conditions.

$\left\{ {\bf T}_m\left(\mathcal{B}\right) \right\}_{m=1}^\infty$ is called uniformly connected if $ {\bf T}_m\left(\mathcal{B}\right)$ is irreducible for all $m\geq 1$ and there is a positive integer $K$ such that for any $m\geq 1$ and any indices pair $(i,j), 1\leq i,j \leq r^m$ where there is $1\leq k \leq K$ such that 
\begin{equation}
 \left({\bf T}_m^k\left(\mathcal{B}\right)\right)_{i,j}\geq 1.
\end{equation}

Uniform connectiveness of $\left\{ {\bf T}_m\left(\mathcal{B}\right) \right\}_{m=1}^\infty$ is equivalent to the existence of finite upper bound $K$ of the diameters of associated graphs of $\left\{ {\bf T}_m\left(\mathcal{B}\right) \right\}_{m=1}^\infty$ , see Section \ref{section3} for details. By Perron-Frobenius Theorem, uniform connectedness implies uniform domination as follows:
\begin{theorem}\label{thm0-2}
	If $\left\{ {\bf T}_m\left(\mathcal{B}\right) \right\}_{m=1}^\infty$ is uniformly connected then $\left\{ {\bf T}_m\left(\mathcal{B}\right) \right\}_{m=1}^\infty$ is uniformly dominated by $\left\{ \rho\left({\bf T}_m\left(\mathcal{B}\right)\right) \right\}_{m=1}^\infty$.
\end{theorem}

Furthermore, the uniform connectedness can be expressed in terms of periodic-mixing conditions as follows.

The shift space $\Sigma \left(\mathcal{B}\right)$ is called horizontal-periodic block gluing if there is an integer $K\geq 1$ such that for any pair of $m$-periodic patterns $\overline{U}_m\in \Sigma_{\infty\times k_1}\left(\mathcal{B}\right)$ and $\overline{V}_m\in \Sigma_{\infty\times k_2}\left(\mathcal{B}\right)$; then $\left( \overline{U}_m,\overline{V}_m\right)$ can be glued together when vertical distance $k\geq K$ for all $m\geq 1$, see the details of Definition \ref{def3.21}. Then, we have the following result for horizontal-periodic block gluing and uniform connectedness.
\begin{theorem}\label{thm0-3}
		If $\left\{  {\bf T}_m \left( \mathcal{B}\right) \right\}_{m=1}^{\infty}$ is irreducible and $\Sigma \left(\mathcal{B}\right)$ is horizontal-periodic block gluing, then $\left\{  {\bf T}_m \left( \mathcal{B}\right) \right\}_{m=1}^{\infty}$ is uniform connected. Furthermore if $\left\{  {\bf T}_m \left( \mathcal{B}\right) \right\}_{m=1}^{\infty}$ is uniform connected and for any $m\geq 1$, there exists an index $i_m$ such that 
	\begin{equation}
	\left( {\bf T}_m\left( \mathcal{B}\right)\right)_{(i_m,i_m)}=1, 
	\end{equation}
	then $\Sigma\left(\mathcal{B}\right)$ is horizontal-periodic block gluing.
\end{theorem}

To study the problems of $h\left(\mathcal{B}\right)=h_{\gamma,p}\left(\mathcal{B}\right)$, we need to the study the transformation of Hermite normal forms $\left[\begin{matrix}
M&L\\0&K
\end{matrix}\right]_{\gamma}$ between $\left[\begin{matrix}
m&\ell\\0&k
\end{matrix}\right]_{\gamma_0}$. Indeed, we have the following result.
\begin{theorem}\label{thm0-4}
		Given $\gamma=\left[\begin{matrix}
	a&b\\
	c&d
	\end{matrix}\right]\in GL_2(\mathbb{Z})$, $\Delta=\det\gamma$,  with $b\neq 0$. Let $k=\gcd(bM,bL+dK)$ and $k=b_1(bM)+b_2(bL+dK)$. Let 
	\begin{equation*}
	bM=m'k
	\end{equation*}
	and
	\begin{equation*}
	bL+dK=\ell'k
	\end{equation*}
	with
	\begin{equation*}
	b_1m'+b_2\ell'=1.
	\end{equation*}
	Then
	\begin{equation}\label{1-12}
	\left[\begin{matrix}
	\frac{m'k}{b}& \frac{-dK+\ell'k}{b}    \\
	0&K
	\end{matrix}\right]_{\gamma}\cong\left[\begin{matrix}
	\frac{m'K}{b}&\frac{ak-\Delta b_2K}{b}\\
	0&k
	\end{matrix}\right]_{\gamma_0},
	\end{equation}
	where all entries in (\ref{1-12}) are integers.
\end{theorem}
Note that two integer $2\times 2$ matrices $A$ and $A'$ is called equivalent and is denoted by $A'\cong A$ if
\begin{equation}
A'\mathbb{Z}^2=A\mathbb{Z}^2.
\end{equation}

Then we have the following result.
\begin{theorem}\label{theorem0-5}
		If $h(\mathcal{B})=h_*(\mathcal{B})$, then $h_{\gamma,p}(\mathcal{B})=h(\mathcal{B})$ for all $\gamma\in GL_2(\mathbb{Z})$. In particular, if $\{ {\bf T}_m(\mathcal{B})  \}_{m=1}^{\infty}$ is uniformly dominated by $\{ \rho\left({\bf T}_m\left(\mathcal{B}\right)\right) \}_{m=1}^{\infty}$, then $h_{\gamma,p}(\mathcal{B})=h(\mathcal{B})$ for all $\gamma\in GL_2(\mathbb{Z})$.
\end{theorem}

The results also hold  when ${\bf T}_m\left(\mathcal{B}\right)$ is not irreducible. Indeed, if $\overline{\bf T}_m\left(\mathcal{B}\right)$ is a maximum irreducible component of ${\bf T}_m\left(\mathcal{B}\right)$, and $\left\{\overline{\bf T}_m\left(\mathcal{B}\right)\right\}_{m=1}^\infty$ is uniformly dominated by $\{\rho\left({\bf T}_m\left(\mathcal{B}\right)\right) \}_{m=1}^{\infty}$, the results of Theorems \ref{thm0-1} and \ref{theorem0-5} hold. Furthermore, the horizontal periodicity can also be replaced by vertical-periodicity which will produced the analogous results. In this case, we consider the vertical-periodic transition matrices $\hat{\bf T}_m\left(\mathcal{B}\right)$ which is obtained by considering the conjugate coordinated system $\hat{\gamma}_0=\left[\begin{matrix}
0&1\\1&0
\end{matrix} \right]$ of ${\gamma}_0=\left[\begin{matrix}
1&0\\0&1
\end{matrix} \right]$.

Finally, a sequence of skew-coordinated system $\gamma_q=\left[\begin{matrix}
1&q\\0&1
\end{matrix} \right],q\geq 1$, is useful in computing the entropy. Indeed, the transformation of Hermite normal form
\begin{equation}
\left[\begin{matrix}
1&0\\0&mq
\end{matrix} \right]_{\gamma_q}\cong\left[\begin{matrix}
m&1\\0&q
\end{matrix} \right]_{\gamma_0}
\end{equation}
holds. Then we can obtain the following results: 
\begin{theorem}\label{thm0-6}
	If $\{ {\bf T}_m(\mathcal{B})  \}_{m=1}^{\infty}$ is irreducible and uniformly dominated by $\{ \rho\left({\bf T}_m\left(\mathcal{B}\right)\right) \}_{m=1}^{\infty}$. Then 
	\begin{equation*}
		h\left(\mathcal{B}\right)=\limsup_{q\to \infty} \log \rho\left( {\bf T}_{\gamma_q,1}\right).
	\end{equation*}
\end{theorem}	
The rest of this paper is organized	as follows. Section \ref{section2} recalls some useful properties of periodic transition matrices ${\bf T}_m\left(\mathcal{B}\right)$. In addition, Section \ref{section3} recalls the Perron-Frobenius Theorem and introduces uniform domination, uniform connectedness and periodic block gluing and proves Theorems \ref{thm0-1}, \ref{thm0-2} and \ref{thm0-3}. Finally, Section \ref{section4} proves the transformations of Hermite normal forms on different skew-coordinates, then proves Theorems \ref{thm0-4}, \ref{theorem0-5} and \ref{thm0-6}.

\section{Entropy and periodic entropy}\label{section2}
In this section, we briefly introduce some properties which concern the entropy of the two-dimensional shift of finite type, see \cite{BL2005 Patterns generation and transition matrices in multi-dimensional lattice models,BHLL2013 Zeta functions for two-dimensional shifts of finite type,BHL2019 Pattern generation problems arising in multiplicative integer systems,BHLL2021 Verification of mixing properties in two-dimensional shifts of finite type,BLL2007 Patterns generation and spatial entropy in two dimensional lattice models}. Let $\mathbb{Z}_{m\times m}$ be the $m\times m$ square lattice in $\mathbb{Z}^2$ and $\mathcal{S}$ be the finite set of symbols (alphabets or colors). The set of all local patterns (or configurations) on $\mathbb{Z}_{m\times m}$ is denoted by $\mathcal{S}^{\mathbb{Z}_{m\times m}}$. A given subset $\mathcal{B} \subseteq \mathcal{S}^{\mathbb{Z}_{m\times m}}$ is called a basic set of admissible(or allowable) local patterns.

$\Sigma_{n\times k}(\mathcal{B})$ is the set of all admissible patterns generated by $\mathcal{B}$ on $\mathbb{Z}_{n\times k}$ and $\Sigma(\mathcal{B})$ is the set of all admissible patterns generated by $\mathcal{B}$ on $\mathbb{Z}^2$. The entropy (spatial entropy) of $\Sigma(\mathcal{B})$ is defined by
\begin{equation}\label{1}
h(\mathcal{B})=\lim_{(n,k)\rightarrow \infty}\frac{1}{nk} \log \left|\Sigma_{n\times k}(\mathcal{B}) \right|.
\end{equation}
Due to the subadditive property of $\log \left|\Sigma_{n\times k}(\mathcal{B})\right|$ in $n$ and $k$, it can be shown the limit (\ref{1}) always exists, see \cite{Chow1996 Pattern formation}.

As for periodic entropy, the set of finite-index subgroup of $\mathbb{Z}^2$ is denoted by $\mathcal{L}_2$. $\mathcal{L}_2$ can be parameterized in Hermite normal forms
\begin{equation*}
\mathcal{L}_2=\left\{  \left[  \begin{matrix}
		n&\ell \\
		0&k
	\end{matrix}  \right]\mathbb{Z}^2 : n\geq 1, k\geq 1\mbox{ and }0\leq \ell \leq n-1
\right\},
\end{equation*}
see \cite{Lind1998 A zeta function for,Mac Duffie1956 The theory of matrices}. Let $\mathcal{P}_{\mathcal{B}}\left(  \left[  \begin{matrix}
n&\ell \\
0&k
\end{matrix}  \right] \right)$ be the set of all $ \left[  \begin{matrix}
n&\ell \\
0&k
\end{matrix}  \right]$-periodic and $\mathcal{B}$-addmissible patterns on $\mathbb{Z}^2$. The number of $\mathcal{P}_{\mathcal{B}}\left(  \left[  \begin{matrix}
n&\ell \\
0&k
\end{matrix}  \right] \right)$ is denoted by $\Gamma_{\mathcal{B}}\left(  \left[  \begin{matrix}
n&\ell \\
0&k
\end{matrix}  \right] \right)$, while the set of all periodic patterns on $\mathbb{Z}^2$, i.e.,
\begin{equation*}
	\begin{aligned}
\mathcal{P}(\mathcal{B})=\bigcup_{n,k\geq 1}\bigcup_{0\leq \ell \leq n-1}\mathcal{P}_{\mathcal{B}}\left(  \left[  \begin{matrix}
	n&\ell \\
	0&k
\end{matrix}  \right] \right)
\end{aligned}
\end{equation*}
is denoted by $\mathcal{P}(\mathcal{B})$. Then the periodic entropy $h_p(\mathcal{B})$ of $\mathcal{P}(\mathcal{B})$ is defined by

\begin{equation}\label{3}
h_p(\mathcal{B})=\limsup_{(n,k)\rightarrow\infty} \sup_{0\leq \ell \leq n-1} \frac{1}{nk} \log \Gamma_{\mathcal{B}}\left(  \left[  \begin{matrix}
	n&\ell \\
	0&k
	\end{matrix}  \right] \right) .
\end{equation}
Unlike spatial entropy, the limit of (\ref{3}) does not exist in general and $\limsup$ is required.
  
It is clear that
\begin{equation}\label{2}
\begin{aligned}
\mathcal{P}(\mathcal{B})\subseteq \Sigma(\mathcal{B}).
\end{aligned}
\end{equation}
Therefore,
\begin{equation}\label{4}
\begin{aligned}
h_p(\mathcal{B})\leq h(\mathcal{B}).
\end{aligned}
\end{equation}
Let $\mathcal{S}_r=\{ 0,1,...,r-1\} $ be the set of symbols. It is known, when $r\geq 5$, for some $\mathcal{B}\subseteq \mathcal{S}_r^{\mathbb{Z}_{2\times 2}}$ have $\mathcal{P}(\mathcal{B})=\emptyset$ and $\Sigma (\mathcal{B})\neq \emptyset$ with $h(\mathcal{B})>0$. Therefore, 
\begin{equation}\label{5}
\begin{aligned}
h_p(\mathcal{B})< h(\mathcal{B}),
\end{aligned}
\end{equation}
see \cite{Culik II1996 An aperiodic set,Kari1996 A small aperiodic}. On the other hand, when $\Sigma(\mathcal{B})$ has some mixing property, then 
\begin{equation}\label{6}
\begin{aligned}
h_p(\mathcal{B})= h(\mathcal{B}),
\end{aligned}
\end{equation}
see \cite{Boyle2010 Multidimensional sofic shifts,Ward1994 Automorphisms}.

In the study of the entropy in \cite{BL2005 Patterns generation and transition matrices in multi-dimensional lattice models,BHLL2013 Zeta functions for two-dimensional shifts of finite type,BLL2007 Patterns generation and spatial entropy in two dimensional lattice models}, a sequence of matrices of patterns ${\bf X}_{2\times n}$ and ${\bf Y}_{m\times 2}$ are introduced to store all patterns in $\Sigma_{2\times n}$ and $\Sigma_{m\times 2}$, respectively. In the following, the two symbols $\mathcal{S}_2=\{0,1 \}$ case is introduced. The general cases $\mathcal{S}_r=\{ 0,1,...,r-1 \}$ with $r\geq 3$ can also be introduced, see \cite{BL2005 Patterns generation and transition matrices in multi-dimensional lattice models,BLL2007 Patterns generation and spatial entropy in two dimensional lattice models,BLL2009 Mixing property and entropy conjugacy of Z2 subshift of finite type: a case study,BHLL2021 Verification of mixing properties in two-dimensional shifts of finite type}. In the case of two symbols, for $\Sigma_{2\times 2}$, a $4\times 4$ matrix ${\bf X}_{2\times 2}=\left[ x_{i_1i_2} \right]$ with $2\times 2$ patterns $x_{i_1i_2}$ as its entries is defined by 
\begin{equation*}
\begin{aligned}
{\bf X}_{2\times 2}&=
\left[\begin{matrix}
\includegraphics[scale=0.3]{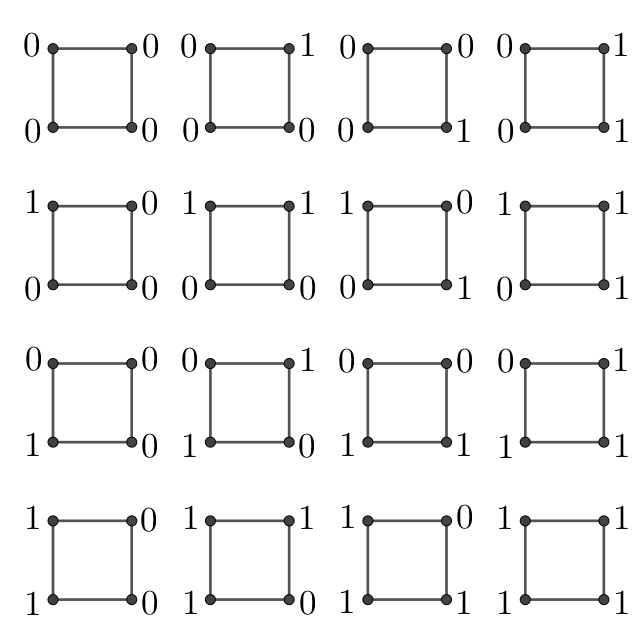}
\end{matrix}\right],
\end{aligned}
\end{equation*}
and ${\bf Y}_{2\times 2}=\left[y_{j_1j_2}\right]$ is defined by
\begin{equation*}
\begin{aligned}{\bf Y}_{2\times 2}=\left[
	\begin{matrix}
\includegraphics[scale=0.3]{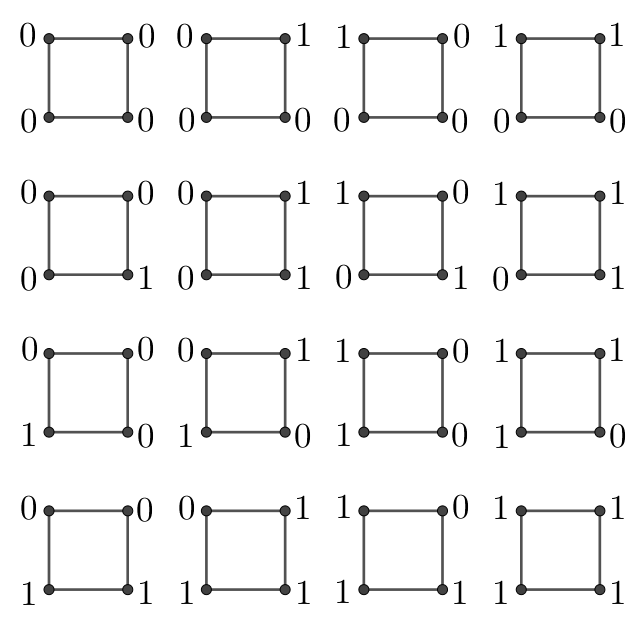}
\end{matrix}\right].
\end{aligned}
\end{equation*}
${\bf X}_{2\times 2}$ is called the horizontal pattern matrix with height 2 and ${\bf Y}_{2\times 2}$ is called the vertical pattern matrix with width 2. Both of them store all $2\times 2$ patterns and the arrangements of these patterns are closely related as follows. Indeed, ${\bf X_{2\times 2}}$ can be presented by $y_{j_1j_2}$ as
\begin{equation*}
\begin{aligned}
{\bf X}_{2\times 2}=\left[
\begin{matrix}
y_{11}&y_{12}&y_{21}&y_{22}\\
y_{13}&y_{14}&y_{23}&y_{24}\\
y_{31}&y_{32}&y_{41}&y_{42}\\
y_{33}&y_{34}&y_{43}&y_{44}
\end{matrix}\right],
\end{aligned}
\end{equation*}
or
\begin{equation}
\begin{aligned}
{\bf X}_{2\times 2}=\left[
\begin{matrix}
X_{2;1}&X_{2;2}\\
X_{2;3}&X_{2;4}
\end{matrix}\right],
\end{aligned}
\end{equation}
with
\begin{equation*}
X_{2;k}=\left[
\begin{matrix}
y_{k1}&y_{k2}\\
y_{k3}&y_{k4}
\end{matrix}\right].
\end{equation*}
 Similarly,
\begin{equation}
	\begin{aligned}
		{\bf Y}_{2\times 2}=\left[
		\begin{matrix}
			x_{11}&x_{12}&x_{21}&x_{22}\\
			x_{13}&x_{14}&x_{23}&x_{24}\\
			x_{31}&x_{32}&x_{41}&x_{42}\\
			x_{33}&x_{34}&x_{43}&x_{44}
		\end{matrix}\right]=
		\left[ \begin{matrix} 
		Y_{2;1}&Y_{2;2}\\
		Y_{2;3}&Y_{2;4}
		\end{matrix} \right].
	\end{aligned}
\end{equation}

Furthermore, all $2\times n$ patterns in $\Sigma_{2\times n}$ can be stored in pattern matrix ${\bf X}_{2\times n}$ with height $n$ which can be defined inductively from ${\bf X}_{2\times 2}$. Indeed, denote 
\begin{equation*}
\begin{aligned}
{\bf X}_{2\times n}=\left[
\begin{matrix}
X_{n;1}&X_{n;2}\\
X_{n;3}&X_{n;4}\\
\end{matrix}\right],
\end{aligned}
\end{equation*}
then

\begin{align}
{\bf X}_{2\times (n+1)}&={\bf X}_{2\times n}\circ \left({\bf E}_{2^{n-1}} \otimes\left[
\begin{matrix}
X_{2;1}&X_{2;2}\\
X_{2;3}&X_{2;4}
\end{matrix}\right]\right)\label{X_2*n+1-1}\\
&={\bf X}_{2\times 2}\circ \left({\bf E}_2 \otimes\left[
\begin{matrix}
X_{n;1}&X_{n;2}\\
X_{n;3}&X_{n;4}
\end{matrix}\right]\right).\label{X_2*n+1-2}
\end{align}
Similarly, for $\Sigma_{m\times 2}$, ${\bf Y}_{m\times 2}$ can be defined recursively with
	\begin{align}
		{\bf Y}_{(m+1)\times 2}
		&={\bf Y}_{m\times 2}\circ \left({\bf E}_{2^{m-1}} \otimes\left[
		\begin{matrix}
			Y_{2;1}&Y_{2;2}\\
			Y_{2;3}&Y_{2;4}\\
		\end{matrix}\right]\right)\label{Y_m+1*2-1}\\
		&={\bf Y}_{2\times 2}\circ \left({\bf E}_2 \otimes\left[
		\begin{matrix}
		Y_{m;1}&Y_{m;2}\\
		Y_{m;3}&Y_{m;4}\\
		\end{matrix}\right]\right),\label{Y_m+1*2-2}
	\end{align}
where ${\bf E}_{2^{n-1}}$ is the $2^{n-1}\times 2^{n-1}$ matrices and entries are 1. $\otimes$ is the Kronecker product (tensor product) and $\circ$ is the Hadamard product i.e., if $A=\left[a_{i,j}\right]_{N\times N}$ and $B=\left[b_{i,j}\right]_{N\times N}$ then $A\circ B=\left[ a_{i,j}b_{i,j} \right]_{N\times N}$. When both $a_{i,j}$ and $b_{i,j}$ are numbers, or matrices of numbers, then $a_{i,j}b_{i,j}$ is the product. When $a_{i,j}$ and $b_{i,j}$ are patterns, it will be understood in the context.

Now, given a basic set of admissible patterns $\mathcal{B}\subseteq \mathcal{S}_2^{\mathbb{Z}_{2\times 2}}$, then the horizontal transition matrix ${\bf H}_n(\mathcal{B})$ can be defined for $\Sigma_{2\times n}(\mathcal{B})$ and vertical transition matrix ${\bf V}_n$ for $\Sigma_{n\times 2}$, respectively. Indeed, ${\bf H}_2$ and ${\bf V}_2$ are defined by
\begin{equation*}
	\begin{aligned}
{\bf H}_2=\left[h_{i_1,i_2} \right]_{4\times 4} \mbox{ and }{\bf V}_2=\left[v_{j_1,j_2} \right]_{4\times 4},
\end{aligned}
\end{equation*}
where
\begin{equation*}
		h_{i_1,i_2}=1 \mbox{ if and only if }x_{i_1i_2}\in \mathcal{B},
\end{equation*}
and	
\begin{equation*}	
		v_{j_1,j_2}=1 \mbox{ if and only if }y_{j_1j_2}\in \mathcal{B}.
\end{equation*}
For $n\geq 3$, ${\bf H}_n$ and ${\bf V}_n$ can be defined recursively by applying (\ref{X_2*n+1-1}) to (\ref{Y_m+1*2-2}) respectively. Indeed,

\begin{align}\label{H_n+1}
{\bf H}_{n+1}&={\bf H}_{n}\circ \left({\bf E}_{2^{n-1}} \otimes\left[
\begin{matrix}
H_{2;1}&H_{2;2}\\
H_{2;3}&H_{2;4}
\end{matrix}\right]\right)\\
&={\bf H}_{2}\circ \left({\bf E}_{2} \otimes\left[
\begin{matrix}
H_{n;1}&H_{n;2}\\
H_{n;3}&H_{n;4}
\end{matrix}\right]\right),
\end{align}
and
\begin{align}\label{V_m+1}
{\bf V}_{m+1}
&={\bf V}_{m}\circ \left({\bf E}_{2^{m-1}} \otimes\left[
\begin{matrix}
V_{2;1}&V_{2;2}\\
V_{2;3}&V_{2;4}
\end{matrix}\right]\right)\\
&={\bf V}_{2}\circ \left({\bf E}_2 \otimes\left[
\begin{matrix}
V_{m;1}&V_{m;2}\\
V_{m;3}&V_{m;4}
\end{matrix}\right]\right),
\end{align}
where
\begin{equation}
{\bf H}_2=\left[
\begin{matrix}
H_{2;1}&H_{2;2}\\
H_{2;3}&H_{2;4}
\end{matrix}\right]
\end{equation}
and
\begin{equation}
{\bf V}_2=\left[
\begin{matrix}
V_{2;1}&V_{2;2}\\
V_{2;3}&V_{2;4}
\end{matrix}\right].
\end{equation}

 Therefore, the entropy of $\Sigma(\mathcal{B})$ can be computed by 
\begin{align}\label{2.11}
h(\mathcal{B})&=\lim_{(n,k)\rightarrow \infty}\frac{1}{nk}\log \left| {\bf H}_n^{k-1}  \right|\\
&=\lim_{(m,\ell)\rightarrow \infty}\frac{1}{m\ell}\log \left| {\bf V}_m^{\ell-1}  \right|.
\end{align}
Furthermore, by Perron-Frobenius theorem, \cite{Gantmacher1960 Matrix Theory,Horn1990 Matrix analysis,Lind1995 An introduction}, 
	\begin{align}
		h(\mathcal{B})&=\limsup_{n\rightarrow \infty}\frac{1}{n}\log \rho({\bf H}_n)\\
		&=\limsup_{m\rightarrow \infty}\frac{1}{m}\log \rho( {\bf V}_m ),
	\end{align}
where $\rho(A)$ is the largest (maximum) eigenvalue of matrix A.

Next, to study the periodic entropy, we need to introduce the matrix of horizontal cylindrical patterns, see \cite{BHLL2013 Zeta functions for two-dimensional shifts of finite type,BLL2007 Patterns generation and spatial entropy in two dimensional lattice models}. Indeed, the $x$-periodic patterns with periodic 2 in $\Sigma_{2\times 2}$ can be stored in cylindrical pattern matrix ${\bf C}_{2\times 2}$ as follows.
\begin{equation}{\bf C}_{2\times 2}=\left[
	\begin{matrix}
		\begin{aligned}
			\includegraphics[scale=0.3]{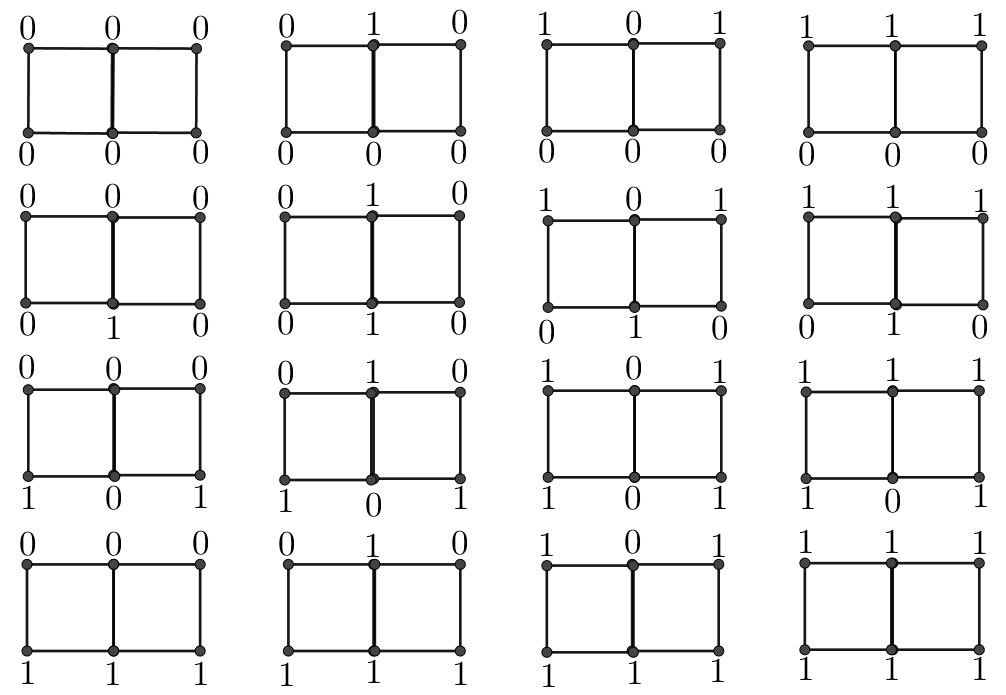}
		\end{aligned}
	\end{matrix}\right].
\end{equation}
${\bf C}_{2\times 2}$ can also be represented in terms of $x_{i_1i_2}$ or $y_{j_1j_2}$. We first introduce the column matrices of patterns $\tilde{\bf X}_{2\times 2}$ of ${\bf X}_{2\times 2}$ and $\tilde{\bf Y}_{2\times 2}$ of ${\bf Y}_{2\times 2}$, respectively,
\begin{equation}
	\begin{aligned}
		\tilde{\bf X}_{2\times 2}&=\left[
		\begin{matrix}		
			x_{11}&x_{21}&x_{12}&x_{22}\\
			x_{31}&x_{41}&x_{32}&x_{42}\\
			x_{13}&x_{23}&x_{14}&x_{24}\\
			x_{33}&x_{43}&x_{34}&x_{44}
		\end{matrix}\right]\\
		&=\left[\begin{matrix}		
			\tilde{X}_{2;1}&\tilde{X}_{2;2}\\
			\tilde{X}_{2;3}&\tilde{X}_{2;4}
		\end{matrix}\right],
	\end{aligned}
\end{equation}
where
\begin{equation*}
	\begin{aligned}
		\tilde{X}_{2;\alpha}=\left[
		\begin{matrix}
			\includegraphics[scale=0.3]{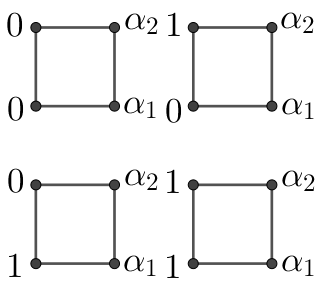}
		\end{matrix}\right],
	\end{aligned}
\end{equation*}
$\alpha=1+2\alpha_1+\alpha_2$ and
\begin{equation}
	\begin{aligned}
		\tilde{\bf Y}_{2\times 2}&=\left[
		\begin{matrix}		
			y_{11}&y_{21}&y_{12}&y_{22}\\
			y_{31}&y_{41}&y_{32}&y_{42}\\
			y_{13}&y_{23}&y_{14}&y_{24}\\
			y_{33}&y_{43}&y_{34}&y_{44}
		\end{matrix}\right]\\
		&=\left[\begin{matrix}		
			\tilde{Y}_{2;1}&\tilde{Y}_{2;2}\\
			\tilde{Y}_{2;3}&\tilde{Y}_{2;4}
		\end{matrix}\right],
	\end{aligned}
\end{equation}
where
\begin{equation*}
	\begin{aligned}
		\tilde{Y}_{2;\beta}=\left[
		\begin{matrix}
			\includegraphics[scale=0.3]{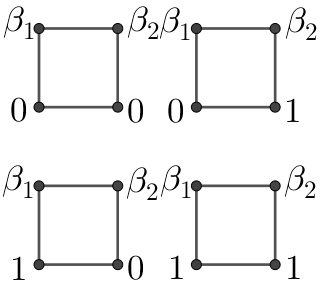}
		\end{matrix}\right],
	\end{aligned}
\end{equation*}
$\beta=1+2\beta_1+\beta_2$.

Then
\begin{equation}
	\begin{aligned}
	{\bf C}_{2\times 2}&=\left[
	\begin{matrix}		
			x_{11}x_{11}&x_{12}x_{21}&x_{21}x_{12}&x_{22}x_{22}\\
			x_{13}x_{31}&x_{14}x_{41}&x_{23}x_{32}&x_{24}x_{42}\\
			x_{31}x_{13}&x_{32}x_{23}&x_{41}x_{14}&x_{42}x_{24}\\
			x_{33}x_{33}&x_{34}x_{43}&x_{43}x_{34}&x_{44}x_{44}
				\end{matrix}\right]\\
				&={\bf Y}_{2\times 2}\circ 	\tilde{\bf X}_{2\times 2}.
		\end{aligned}
\end{equation}
Similarly, the $y$-periodic patterns with period 2 in $\Sigma_{2\times 2}$ can be stored in vertical cylindrical pattern matrix $	\hat{{\bf C}}_{2\times 2}$ with
\begin{equation}
\hat{{\bf C}}_{2\times 2}={\bf X}_{2\times 2}\circ 	\tilde{\bf Y}_{2\times 2}.
\end{equation}
The horizontal cylindrical pattern matrix which consists of patterns of $x$-periodic with period n and high 2, defined by
\begin{equation}\label{C_n2}
	\begin{aligned}
		{\bf C}_{n\times 2}&={\bf Y}_{n\times 2}\circ \left[ \begin{matrix}
		{\bf E}_{2^{n-2}}\otimes \tilde{X}_{2;1} & 	{\bf E}_{2^{n-2}}\otimes \tilde{X}_{2;2}\\
		{\bf E}_{2^{n-2}}\otimes \tilde{X}_{2;3} & 	{\bf E}_{2^{n-2}}\otimes \tilde{X}_{2;4}
		\end{matrix}\right],
	\end{aligned}
\end{equation}
and the vertical cylindrical matrix which consists of patterns of $y$-periodic with period n and wide 2 can be defined by
\begin{equation}\label{hatC_2n}
	\begin{aligned}
		\hat{\bf C}_{2\times n}&={\bf X}_{2\times n}\circ \left[ \begin{matrix}
			{\bf E}_{2^{n-2}}\otimes \tilde{Y}_{2;1} & 	{\bf E}_{2^{n-2}}\otimes \tilde{Y}_{2;2}\\
			{\bf E}_{2^{n-2}}\otimes \tilde{Y}_{2;3} & 	{\bf E}_{2^{n-2}}\otimes \tilde{Y}_{2;4}
		\end{matrix}\right].
	\end{aligned}
\end{equation}

Given a basic set of admissible patterns $\mathcal{B}$, then the horizontal cylindrical (or horizontal-periodic) transition matrix ${\bf T}_n(\mathcal{B})$ of ${\bf C}_{n\times 2}$ and vertical cylindrical (or vertical-periodic) matrix $\hat{\bf T}_n(\mathcal{B})$ of $\hat{\bf C}_{n\times 2}$ can be defined by
\begin{equation}
{\bf T}_2(\mathcal{B})={\bf V}_2\circ \tilde{\bf H}_2
\end{equation}
and
\begin{equation}
\hat{\bf T}_2(\mathcal{B})={\bf H}_2\circ \tilde{\bf V}_2,
\end{equation}
where $\tilde{\bf H}_2$ and $\tilde{\bf V}_2$ are column matrices of ${\bf H}_2$ and ${\bf V}_2$ respectively. Furthermore, from (\ref{C_n2}) and (\ref{hatC_2n}), we have
\begin{equation}\label{T_n}
	\begin{aligned}
		{\bf T}_n={\bf V}_n\circ \left[ \begin{matrix}
			{\bf E}_{2^{n-2}}\otimes \tilde{H}_{2;1} & 	{\bf E}_{2^{n-2}}\otimes \tilde{H}_{2;2}\\
			{\bf E}_{2^{n-2}}\otimes \tilde{H}_{2;3} & 	{\bf E}_{2^{n-2}}\otimes \tilde{H}_{2;4}
		\end{matrix}\right],
			\end{aligned}
	\end{equation}
and	
\begin{equation}\label{hatT_n}
			\begin{aligned}
		\hat{\bf T}_n={\bf H}_n\circ \left[ \begin{matrix}
			{\bf E}_{2^{n-2}}\otimes \tilde{V}_{2;1} & 	{\bf E}_{2^{n-2}}\otimes \tilde{V}_{2;2}\\
			{\bf E}_{2^{n-2}}\otimes \tilde{V}_{2;3} & 	{\bf E}_{2^{n-2}}\otimes \tilde{V}_{2;4}
		\end{matrix}\right].
	\end{aligned}
\end{equation}

Next, we recall the periodic patterns in $\mathbb{Z}^2$. A global pattern ${\bf U}=\left[  U_{i,j}  \right]$ on $\mathbb{Z}^2$ is called $\left[\begin{matrix}
n&\ell\\
0&k
\end{matrix}\right]$-periodic if it satisfies
\begin{equation}
U_{i+rn+s\ell,j+sk}=U_{i,j}
\end{equation}
for all $(i,j)\in \mathbb{Z}^2$ and $r,s \in \mathbb{Z}$. Note that $\left[\begin{matrix}
n&0\\
0&k
\end{matrix}\right]$-periodic is just the rectangular $(n,k)$-periodic. To study these periodic patterns, we need to recall the rotational matrix ${\bf R}_n, n\geq 1$, see \cite{BHLL2013 Zeta functions for two-dimensional shifts of finite type}. 

The shift (to the left) $\sigma$ of any $n$-sequence $(u_1u_2\cdots u_n)\in \{ 0,1\}^{n}$ is defined by
\begin{equation}
\sigma(u_1u_2\cdots u_n)=(u_2u_3\cdots u_nu_1).
\end{equation}
Notably, the shift of any one-dimensional periodic sequence $(u_1u_2\cdots u_n)^{\infty}=(u_1u_2\cdots u_n u_1u_2\cdots)$ with period $n$ becomes $(u_2u_3\cdots u_nu_1)^\infty=(u_2u_3\cdots u_nu_1u_2u_3\cdots)$. The sequence $(u_1u_2\cdots u_n)$ and its shift $(u_2\cdots u_nu_1)$ can also be related by their counting number $\chi(u_1u_2\cdots u_n)$. Define
\begin{equation*}
	\begin{aligned}
	i=\chi(u_1u_2\cdots u_n)=1+\sum_{j=1}^n u_j 2^{n-j},
	\end{aligned}
\end{equation*}
and $n$-shift map
\begin{equation}
		\sigma(i)\equiv\sigma_n(i)=\chi(u_2u_3\cdots u_nu_1).
\end{equation}
 It is easy to see that $\sigma$ is a bijection map on $\left\{1,2,...,2^n\right\}$. For any $n\geq 1$, the $2^n \times 2^n$ rotational matrix ${\bf R}_n=\left[ R_{n;i,j} \right], R_{n;i,j}\in \{ 0,1 \}$, is defined by
\begin{equation}
		R_{n;i,j}=1\mbox{ if and only if } j=\sigma(i).
\end{equation}
${\bf R}_n$ is a permutation matrix. Indeed, ${\bf R}_n$ can be written explicitly as follows.
\begin{proposition}\label{Rotation}
	For any integer $n\geq 2$,
	\begin{equation}
	\left\{	\begin{aligned}
	&R_{n;i,2i-1}=1 \mbox{ \rm and }R_{n;2^{n-1}+i,2i}=1& &\mbox{\rm, if }1\leq i \leq 2^{n-1}, \\
	&R_{n;i,j}=0&                     &\mbox{\rm, otherwise.} 
		\end{aligned}\right.
\end{equation}
Equivalently,
	\begin{equation}
	\sigma(i)=\left\{\begin{aligned}
			&2i-1           &  &\mbox{\rm, if }1\leq i \leq 2^{n-1}, \\
			&2(i-2^{n-1})&  &\mbox{\rm, if }2^{n-1}<i\leq 2^n. 
	\end{aligned}\right.
\end{equation}
Furthermore, $({\bf R}_n^j)_{i,\sigma^j(i)}=1$ for any $1\leq j \leq n-1$ and ${\bf R}_n^n=I_{2^n}$ where $I_N$ is the $N\times N$ identity matrix.
\end{proposition}	

We recall the following results in \cite{BHLL2013 Zeta functions for two-dimensional shifts of finite type}.
\begin{proposition}For $n\geq 2$, ${\bf T}_n=\left[  T_{n;i,j}  \right]$ is ${\bf R}_n$-symmetric, i.e., 
	\begin{equation*}
	\begin{aligned}
T_{n;\sigma^{\ell}(i),\sigma^{\ell}(j)}=T_{n;i,j}
	\end{aligned}
\end{equation*}	
for all $1\leq i,j \leq 2^n$ and $0\leq \ell \leq n-1$.	 
\end{proposition}

\begin{proposition}\label{gamma n l o k }
For $n,k\geq 1$ and $0\leq \ell \leq n-1$,
\begin{equation}
		\Gamma_{\mathcal{B}}\left(  \left[  \begin{matrix}
		n&\ell \\
		0&k
		\end{matrix}  \right] \right)=tr\left( {\bf T}_n^k{\bf R}_n^\ell  \right),
\end{equation}		
and
 \begin{equation}
h_p(\mathcal{B})=\limsup_{(n,k)\rightarrow \infty} \sup_{0\leq \ell \leq n-1} \frac{1}{nk} \log tr\left( {\bf T}_n^k{\bf R}_n^\ell  \right).
\end{equation}	
\end{proposition}
It is worth investigating the periodic patterns in more detail. Indeed, for any integer $\ell$, define the $\ell$-shift periodic entropy $h_{\ell}(\mathcal{B})$ by
\begin{equation}
h_{\ell}(\mathcal{B})=\limsup_{(n,k)\rightarrow\infty}\frac{1}{nk}\log \Gamma_{\mathcal{B}}\left(  \left[  \begin{matrix}
n&\ell \\
0&k
\end{matrix}  \right] \right).
\end{equation}
In particular, when $\ell=0$, $h_0(\mathcal{B})$ is the rectangular periodic entropy, i.e.,
\begin{equation*}
h_0(\mathcal{B})=\limsup_{(n,k)\rightarrow\infty}\frac{1}{nk}\log \Gamma_{\mathcal{B}}\left(  \left[  \begin{matrix}
n&0 \\
0&k
\end{matrix}  \right] \right).
\end{equation*}
For $\ell=1$, $h_1(\mathcal{B})$ is also very useful which will be studied later by using the skew-coordinated system.

It is clear that 
\[
h_{\ell}(\mathcal{B})\leq h_p(\mathcal{B})
\] 
for all $\ell$. The inequalities may be hold for some $\ell$. Indeed, consider the 3 symbols basic set $\mathcal{B}_{3,0}$ defined by
\begin{equation*}
\mathcal{B}_{3,0}=
\left\{
\begin{aligned}
\includegraphics[scale=0.3]{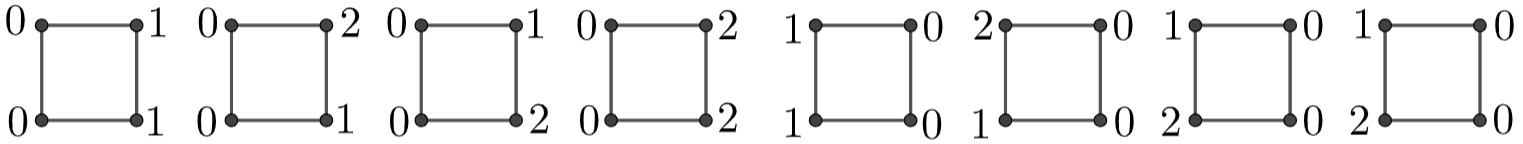}
\end{aligned}
\right\}.
\end{equation*}
It is easy to see 
\begin{equation*}
\Gamma_{\mathcal{B}_{3,0}} \left( \left[ \begin{matrix}
2m&2\ell\\
0&k
\end{matrix}\right] \right)=2^{mk},
\end{equation*}
and
\begin{equation*}
\Gamma_{\mathcal{B}_{3,0}} \left( \left[ \begin{matrix}
2m&2\ell+1\\
0&k
\end{matrix}\right] \right)=0.
\end{equation*}
Therefore,
\begin{equation*}
h_{2\ell}(\mathcal{B}_{3,0})=\frac{1}{2}\log 2 ,h_{2\ell+1}(\mathcal{B}_{3,0})=0,
\end{equation*}
and
\[
h_p(\mathcal{B}_{3,0})=\frac{1}{2}\log 2.
\]
  The problem of $h_{\ell}(\mathcal{B})=h_p(\mathcal{B})$
 will be studied in Section \ref{section4} by using the skew-coordinated system.
\section{Uniform Controls and Connects}\label{section3}
In this section, we will study entropy and periodic entropy and introduce some conditions on $\mathcal{B}$ which implies $h_p(\mathcal{B})=h(\mathcal{B})$.

We begin with recalling some results from the matrix theory, in particular, from the Perron-Frobenius Theorem and its applications, see \cite{Gantmacher1960 Matrix Theory,Horn1990 Matrix analysis}.

A $N\times N$ matrix $A=[a_{i,j}]_{N\times N}$ is non-negative if $a_{i,j}\geq 0$ for all $1\leq i,j \leq N$ and integral if all its entries are integers. 

$\lambda$ is an eigenvalue of $A$ with right (column) eigenvector $ v=(v_1, ..., v_N)^t$ if 
\begin{equation}
Av=\lambda v,
\end{equation}
and left (row) eigenvector $u=(u_1, ..., u_N)$ of $A$ with eigenvalue $\lambda$ if 
\begin{equation}
uA=\lambda u.
\end{equation}

A matrix $A=[a_{i,j}]$ is called irreducible if for any $1\leq i,j \leq N$, there is some positive integer $k$ such that $(A^k)_{i,j}>0$. A is called primitive (irreducible and aperiodic) if there is some positive integer $k$ such that $(A^k)_{i,j}>0$ for all $1\leq i,j \leq N$. It is clear that if $A$ is primitive then $A$ is irreducible.

Now, we state the Perron-Frobenius Theorem.
\begin{theorem}[Perron-Frobenius Theorem]\label{theorem 4.1}
	\item[(i)]Let $A$ be an $N\times N$, $N\geq 2$, irreducible matrix. Then there is the positive maximum eigenvalue $\rho=\rho(A)$ which is algebraic simple and the corresponding eigenvectors $v$ and $u$ are positive (i.e., $v_i>0$ and $u_i>0$ for all $i$.) Any other eigenvalue $\lambda$ of $A$, $\left| \lambda \right|\leq \rho$, $\rho$ is the only eigenvalue with a non-negative eigenvector.
	\item[(ii)]If $A$ is primitive then for each $1\leq i,j \leq N $,
	\begin{equation}\label{4.3}
	\lim_{k \rightarrow \infty} \frac{(A^k)_{i,j}}{\rho^k}= v_i u_j,
	\end{equation}
	and
	\begin{equation}\label{4.4}
	\left| \lambda  \right| < \rho
	\end{equation}	
	for any other eigenvalue $\lambda$ of $A$.
\end{theorem}

The following results of the structures of irreducible and reducible matrices are very useful, see \cite{Horn1990 Matrix analysis}.

\begin{theorem}\label{theorem 4.2}
	\item[(i)]If $A$ is an irreducible $N\times N$ matrix, then there is positive integer $p\geq 1$ and a permutation $P$ such that
	\begin{equation}\label{4.5}
	PAP^t=\left[  \begin{matrix}
	O&B_1&O&\cdots&O\\
	&O&&&\vdots\\
	&&\ddots&&O\\
	&&&O&B_{p-1}\\
	B_p&&&&O
	\end{matrix}  \right]
	\end{equation} 
	and
	\begin{equation}\label{4.6}
	PA^{kp}P^t=\left[
	\begin{matrix}
	A_1^k&&O\\
	&\ddots&\\
	O&& A_p^k
	\end{matrix}
	\right],
	\end{equation}
	where $A_i=B_i \cdots B_p B_1 \cdots B_{i-1}$ and $A_i$ is primitive for any $1\leq i \leq p$.
	\item[(ii)]If $A$ is a reducible $N\times N$ matrix then under a permutation, $A$ can be reduced to a normal reducible form : 
	\begin{equation}\label{4.7}
	A=\left[
	\begin{matrix}
	A_1&O&\cdots&&&&&&O\\
	O&A_2&O&\cdots&&&&&O\\
	\vdots&&&&&&&&\vdots\\
	O&\cdots&&O&A_g&O&\cdots&&O\\
	A_{g+1,1}&\cdots&&&A_{g+1,g}&A_{g+1}&O&\cdots&O\\
	&&&&&&&&\\
	A_{s,1}&\cdots&&&A_{s,g}&\cdots&&A_{s,s-1}&A_{s}	
	\end{matrix}
	\right],
	\end{equation}
	where $1\leq g \leq s$, $A_i$, $ 1\leq i \leq g$ are irreducible matrices, and
	\begin{equation}
	\rho(A)= \max_{1\leq i \leq s} \rho(A_i).
	\end{equation}	
\end{theorem}

From Theorems \ref{theorem 4.1} and \ref{theorem 4.2}, we can obtain the asymptotic result for $(A^k)_{i,j}$ as $k\rightarrow \infty$.

\begin{theorem}\label{thm 4.3}
	Let $A=[a_{i,j}]_{N \times N}$ be an irreducible non-negative integral matrix with maximum eigenvalue $\rho=\rho(A)>0$ and cycle $p\geq 1$. Then for any $1\leq i,j \leq N$, there exists a $k=k(i,j)$ such that 
	\begin{equation}\label{4.*}
	c_1 \rho^{\alpha p} \leq (A^{\alpha p+k})_{i,j} \leq d_1 \rho^{\alpha p}
	\end{equation}
	for large $\alpha \geq 1$ where $0< c_1 \leq d_1$ are positive constants only depend on $A$.
	
	Furthermore,
	\begin{equation}\label{4.9}
	\limsup_{k\rightarrow \infty} \frac{1}{k}\log \left| A^k \right|=\log \rho(A)
	\end{equation}
	and
	\begin{equation}\label{4.10}
	\limsup_{k\rightarrow \infty} \frac{1}{k}\log tr(A^k)= \log \rho(A).
	\end{equation}
\end{theorem}	

\begin{proof}
	From (\ref{4.6}), 
	\begin{equation*}
	A^{\alpha p+k}= P (\overline{A})^\alpha P^t A^k
	\end{equation*}
	where 
	\begin{equation*}
	\overline{A}=\left[
	\begin{matrix}
	A_1&&O\\
	&\ddots&\\
	O&& A_p
	\end{matrix}
	\right], P=[P_{i,j}]_{N \times N}, P^t=[P^t_{i,j}]_{N\times N},
	\end{equation*}
	and $A_k$ is primitive and $\rho(A_k)=\rho(A)$ for all $1\leq k \leq p $.
	
	For each $1\leq k \leq p $, let $v_k=(v_{k,1}, ... , v_{k,N_k})$ and $u_k=(u_{k,1}, ..., u_{k,N_k})$ are right and left positive eigenvectors of $A_k$. Note that $\sum_{k=1}^p N_k =N$.
	
	Then (\ref{4.3}) implies
	\begin{equation}\label{4.**}
	\lim_{\alpha \rightarrow \infty} \frac{(A_k^{\alpha})_{s,t}}{\rho^{\alpha p}}= v_{k,s} u_{k,t},
	\end{equation}
	for $1\leq s,t \leq N_k$. 
	
	Then
	\begin{equation}\label{4.***}
	(A^{\alpha p+\beta})_{i,j}=\sum_{i_2,i_3=1}^{N} P_{i i_2} (\overline{A}^\alpha)_{i_2 ,i_2} P^t_{i_2, i_3} (A^\beta)_{i_3, j}.
	\end{equation}	
	
	Since $P$ is a permutation, there is a unique $i_2$ such that 
	\begin{equation}\label{4.****}
	P_{i, i_2}=1.
	\end{equation}
	For large $\alpha$, (\ref{4.**}) implies 
	\begin{equation}\label{4.*****}
	\lim_{\alpha \rightarrow \infty} \frac{(\overline{A}^\alpha)_{i_2, i_2}}{\rho^{\alpha p}}=v_{k,s} u_{k,s}
	\end{equation}
	for some $k \geq 1$ and $s \geq 1$.
	
	Since $P^t$ is permutation, there exists a unique $i_3$ such that $P^t_{i_2 ,i_3}=1$. Therefore, choose $k \geq 1$ such that 
	\begin{equation}\label{4.******}
	(A^k)_{i j}\geq 1,
	\end{equation}
	where $k$ only depends on $i$ and $j$.
	
	By (\ref{4.***}), (\ref{4.****}), (\ref{4.*****}) and (\ref{4.******}), there are constants $0< c_1 \leq d_1$ only depend on $A_k, 1\leq k \leq p$ (and so $A$), such that
	\begin{equation*}
	c_1 \rho^{\alpha p} \leq (A^{\alpha p+k})_{i,j} \leq d_1 \rho^{\alpha p}.
	\end{equation*}
	The proof of (\ref{4.*}) is complete. (\ref{4.9}) and (\ref{4.10}) follow easily from (\ref{4.*}).
\end{proof}

The following lemma concerning the double limit and iterated limits of double sequences are useful.

\begin{lemma}\label{lemma 4.4}
	For any double sequence $a_{m,n}$, we have
	\begin{equation}\label{4.11}
	\limsup_{m\rightarrow\infty} \left(\limsup_{n\rightarrow\infty}   a_{m,n} \right) \leq \limsup_{(m,n)\rightarrow\infty} a_{m,n},
	\end{equation}
	and
	\begin{equation}\label{4.12}
	\limsup_{n\rightarrow\infty} \left(\limsup_{m\rightarrow\infty}   a_{m,n} \right) \leq \limsup_{(m,n)\rightarrow\infty} a_{m,n}.
	\end{equation}
	In the case that the double limit exists, then the equalities in (\ref{4.11}) and (\ref{4.12}) hold.
\end{lemma}

The following results for ${\bf H}_n$ and ${\bf T}_n$ (or ${\bf V}_m$ and $\hat{\bf T}_m$) are useful in the study of entropy and periodic entropy.
\begin{lemma}\label{lemma 4.5}
	For any basic set of admissible local patterns $\mathcal{B}$, 
	\begin{equation}\label{4.13}
	tr({\bf H}_n^m(\mathcal{B}))=\left|   {\bf T}_m^{n-1} (\mathcal{B})  \right|.
	\end{equation}
\end{lemma}

\begin{proof}
	Both sides of (\ref{4.13}) are the cardinal numbers of $m\times n$ $\mathcal{B}$-admissible patterns which are horizontal periodic with period $m$ and height $n$. The proof is complete.
\end{proof}

Furthermore, we have the following results.
\begin{lemma}\label{lemma 4.6-1}
	For any basic set of admissible local patterns $\mathcal{B}$, 
	\begin{equation}\label{4.18-1}
	h(\mathcal{B})=\limsup_{(m,n)\rightarrow\infty} \frac{1}{mn}\log \left| {\bf T}_m^n(\mathcal{B}) \right|.
	\end{equation}
\end{lemma}

\begin{proof}
	Since the double limit of $h(\mathcal{B})$ holds in (\ref{2.11}) and equals to iterated limits, we have
	\begin{equation}\label{4.20}
	\begin{aligned}
	h(\mathcal{B})&=	\lim_{(m,n)\rightarrow\infty} \frac{1}{mn}\log \left|{\bf H}_n^{m-1}(\mathcal{B})\right|\\
	&=	\lim_{n\rightarrow\infty}\frac{1}{n} \left( \lim_{n\rightarrow\infty}\frac{1}{m}\log \left|{\bf H}_n^{m-1}(\mathcal{B})\right|  \right)\\
	&=	\limsup_{n\rightarrow\infty}\frac{1}{n} \log \rho\left( {\bf H}_n(\mathcal{B}) \right).
	\end{aligned}
	\end{equation}	
	On the other hand, by (\ref{4.13}) of Lemma \ref{lemma 4.5}, we have
	\begin{equation*}
	\begin{aligned}
	\limsup_{(m,n)\rightarrow\infty} \frac{1}{mn}\log \left| {\bf T}_m^n(\mathcal{B}) \right|&= \limsup_{(m,n)\rightarrow\infty} \frac{1}{mn}\log tr\left( {\bf H}_{n+1}^m(\mathcal{B}) \right)\\
	&\geq \limsup_{n\rightarrow\infty}\frac{1}{n} \left( \lim_{m\rightarrow\infty}\frac{1}{m}\log tr\left( {\bf H}_{n+1}^m(\mathcal{B}) \right)   \right)\\
	&=	\limsup_{n\rightarrow\infty}\frac{1}{n} \log \rho\left( {\bf H}_{n+1}(\mathcal{B}) \right)\\
	&=	\limsup_{n\rightarrow\infty}\frac{1}{n+1} \log \rho\left( {\bf H}_{n+1}(\mathcal{B}) \right).
	\end{aligned}
	\end{equation*}
	Therefore, we have
	\begin{equation}\label{4.21}
	h(\mathcal{B})\leq\limsup_{(m,n)\rightarrow\infty} \frac{1}{mn}\log \left| {\bf T}_m^n(\mathcal{B}) \right|.
	\end{equation}
	By (\ref{2.11}) and (\ref{4.15}), the equality holds in (\ref{4.21}). That is, (\ref{4.20}) holds. The proof is complete.
\end{proof}
For any basic set of admissible patterns $\mathcal{B}$, denote by
\begin{equation}
h_*(\mathcal{B})=\limsup_{m \rightarrow \infty}\frac{1}{m}\log\rho({\bf T}_m(\mathcal{B})).
\end{equation}
Then we have the following results.
\begin{lemma}\label{lemma 4.6}
	For any basic set of admissible local patterns $\mathcal{B}$, 
	\begin{equation}\label{4.16}
	h_*(\mathcal{B}) \leq h_p(\mathcal{B})\leq h(\mathcal{B}).
	\end{equation}
\end{lemma}

\begin{proof}
	From (\ref{4.13}), we have 
	\begin{equation}\label{4.15}
	\left|   {\bf T}_m^{n-1} (\mathcal{B})  \right|\leq \left| {\bf H}_n^m(\mathcal{B}) \right|.
	\end{equation}
	Therefore, by (\ref{4.11}) of Lemma \ref{lemma 4.4}, (\ref{4.9}) of Theorem \ref{thm 4.3} and (\ref{4.15}),
	\begin{equation*}
	\begin{aligned}
	\limsup_{m\rightarrow\infty} \frac{1}{m}\log \rho({\bf T}_m(\mathcal{B}))
	& \leq 	\limsup_{(m,n)\rightarrow\infty} \frac{1}{mn}\log \left|{\bf T}_m^{n-1}(\mathcal{B})\right| \\
	& \leq 	\limsup_{(m,n)\rightarrow\infty} \frac{1}{mn}\log \left|{\bf H}_m^{n-1}(\mathcal{B})\right| = h(\mathcal{B}).
	\end{aligned}
	\end{equation*}
	As for (\ref{4.16}),
	\begin{equation*}
	\begin{aligned}
	h_p(\mathcal{B})&=\limsup_{(m,k)\rightarrow\infty} \sup_{0\leq \ell \leq m-1} \frac{1}{mk}\log \Gamma_{\mathcal{B}}\left(  \left[  \begin{matrix}
	m&\ell \\
	0&k
	\end{matrix}  \right] \right) \\
	&=\limsup_{(m,k)\rightarrow\infty} \sup_{0\leq \ell \leq m-1} \frac{1 }{mk}\log tr\left(  {\bf T}_m^k(\mathcal{B}) {\bf R}_m^\ell  \right)\\
	&\geq \limsup_{(m,k)\rightarrow\infty}  \frac{1 }{mk}\log tr\left(  {\bf T}_m^k(\mathcal{B})  \right)\\	
	&\geq \lim_{m\rightarrow\infty} \frac{1}{m}\left( \limsup_{k\rightarrow\infty}  \frac{1}{k}\log tr\left(  {\bf T}_m^k(\mathcal{B})  \right) \right)\\	
	&=\limsup_{m\rightarrow\infty} \frac{1}{m}\log \rho({\bf T}_m(\mathcal{B})).
	\end{aligned}	
	\end{equation*}
	The proof is complete.
\end{proof}

The equality of (\ref{4.16}) does not hold for some $\mathcal{B}$ as mentioned in \cite{Culik II1996 An aperiodic set,Kari1996 A small aperiodic}. However, for some classes of basic set $\mathcal{B}$, it holds. In \cite{BLL2007 Patterns generation and spatial entropy in two dimensional lattice models}, using (\ref{4.18-1}) it has been proved that when ${\bf H}_2(\mathcal{B})$ is symmetric, then 
\begin{equation*}
    h(\mathcal{B})\leq \frac{1}{2m} \log \rho({\bf T}_{2m})
\end{equation*}
for all $m\geq 1$. Hence, $h(\mathcal{B})=h_*(\mathcal{B})$. In viewing (\ref{4.18-1}), we can introduce another classes of basic sets. We need the following notion.

\begin{definition}
	For basic set $\mathcal{B}$, the sequence of cylindrical transition matrices $\{ {\bf T}_m (\mathcal{B})  \}_{m=1}^{\infty}$ is called \emph{uniformly dominated} (or uniformly controlled or uniformly bounded) by their maximum eigenvalues $\{ \rho \left( {\bf T}_m (\mathcal{B}) \right)  \}_{m=1}^{\infty}$ if there is a positive function $c(m,k)$ such that 
	\begin{equation}\label{4.17}
	\left|  {\bf T}_m^k (\mathcal{B})  \right| \leq c(m,k) \rho ({\bf T}_m)^k
	\end{equation}
	with
	\begin{equation}\label{4.17-1}
	\lim_{(m,k) \rightarrow \infty} \frac{1}{mk} \log c(m,k)=0
	\end{equation}
	for all $m\geq 1$ and $k\geq 1$.
\end{definition}

We show that the equality holds in (\ref{4.16}) when basic set $\mathcal{B}$ satisfies the uniformly dominant properties (\ref{4.17}) and (\ref{4.17-1}).

\begin{theorem}\label{thm 4.9}
	If (\ref{4.17}) and (\ref{4.17-1}) hold, then
	\begin{equation}\label{4.18}
	h(\mathcal{B})=h_p(\mathcal{B})=h_*(\mathcal{B}).
	\end{equation}
\end{theorem}

\begin{proof}
	By Lemma \ref{lemma 4.6}, it suffices to prove 
	\begin{equation}\label{4.19}
	h(\mathcal{B})\leq \limsup_{m\rightarrow\infty} \frac{1}{m}\log \rho({\bf T}_m(\mathcal{B})).
	\end{equation}
	Then the condition of uniform controllability (\ref{4.17}) and (\ref{4.17-1}) imply
	\begin{equation*}
	\begin{aligned}
	\limsup_{(m,n)\rightarrow\infty} \frac{1}{mn}\log \left| {\bf T}_m^n(\mathcal{B}) \right| &\leq 	
	\limsup_{(m,n)\rightarrow\infty} \frac{1}{mn}\log C(m,n) + \limsup_{(m,n)\rightarrow\infty} \frac{1}{m}\log \rho ({\bf T}_m(\mathcal{B}))\\
	&=\limsup_{m\rightarrow\infty} \frac{1}{m}\log \rho ({\bf T}_m(\mathcal{B})).
	\end{aligned}
	\end{equation*}	
	Therefore, (\ref{4.19}) holds by Lemma \ref{lemma 4.6-1} and then (\ref{4.18}) follows. The proof is complete.
\end{proof}

To study the uniformly dominant condition (\ref{4.17}), we will introduce some notion and results. We first recall the following lemma, see \cite{Lind1995 An introduction}.

\begin{lemma}\label{lemma 4.10}
	If the non-negative $N\times N$ matrix $A=[a_{i,j}]$ has positive eigenvalue $\rho(A)$ with positive eigenvector $v=(v_1,...,v_N)^t$. Then
	\begin{equation}\label{4.22}
	\frac{c}{d} N \rho(A)^k \leq \left| A^k \right| \leq 	\frac{d}{c} N \rho(A)^k,
	\end{equation}
	where
	\begin{equation}\label{4.23}
	c=\min_{i} v_i \mbox{ and } d=\max_i v_i.
	\end{equation}
\end{lemma}

\begin{proof}
	Since $Av=\rho v$, $A^k v = \rho^k v$. Therefore, $\sum_{j=1}^N (A^k)_{i,j} v_j = \rho^k v_i$ for all $i$. Hence, 
	\begin{equation}\label{4.24}
	c\rho^k \leq \sum_{j=1}^N (A^k)_{i,j} v_j  \leq d\rho^k
	\end{equation}
	for all $1\leq i \leq N$. From the right hand side of (\ref{4.24})
	\begin{equation*}
	\sum_{j=1}^N (A^k)_{i,j}  \leq \frac{d}{c}\rho^k.
	\end{equation*}
	Hence
	\begin{equation*}
	\left| A^k \right| \leq 	\frac{d}{c} N \rho(A)^k.
	\end{equation*}
	From the left hand side of (\ref{4.24}), we have $\frac{c}{d}\rho^k\leq 	\sum_{j=1}^N (A^k)_{i,j}$. Hence $\frac{c}{d} N \rho(A)^k \leq \left| A^k \right|$. The proof is complete.
\end{proof}

The following lemma estimates the ratio of $\frac{d}{c}$.
\begin{lemma}
	Let $N \times N$ non-negative integral matrix $A=[a_{i,j}]$ have positive eigenvalue $\rho\geq 1$ with positive eigenvector $v=(v_1,..., v_N)$. If there is a positive constant $K\geq 1$ such that for any $1\leq i,j \leq N$, there is $1\leq k \leq K$ with $k=k(i,j)$ such that 
	\begin{equation}\label{4.25}
	(A^k)_{i,j} \geq 1.
	\end{equation}
	Then 
	\begin{equation}\label{4.26}
	\frac{d}{c}:=\frac{\max_i v_i}{\min_i v_i } \leq \rho^K.
	\end{equation}
\end{lemma}

\begin{proof}
	For any $1 \leq i,j \leq N $, $(A^k)_{i,j} \geq 1$ and $\sum_{\ell=1}^N (A^k)_{i\ell} v_\ell = \rho^k v_i$ imply
	\begin{equation*}
	v_j \leq (A^k)_{i,j} v_j \leq \rho^k v_i	.
	\end{equation*}
	Therefore, $v_j \leq \rho^k v_i$, and then $v_j \leq \rho^K v_i$. Hence (\ref{4.26}) holds.	
\end{proof}

We introduce some concepts from graph theory. For any matrix $A=[a_{i,j}]_{N\times N}$, with $a_{i,j}\in\{ 0,1\}$, then the associated graph $G=G(A)$ is defined by the vertex set $V=\{1,2,...,N\}$ and for $i,j \in V$ there is an edge from $i$ to $j$ if $a_{i,j}=1$. If there is $k$ such that $(A^k)_{i,j}\geq 1$, then the distance $d(i,j)$ of vertices $i$ and $j$ is defined by
\begin{equation}
d(i,j)=\min \left\{ k : (A^k)_{i,j}\geq 1 \right\}.
\end{equation}  
Otherwise, if vertices $i$ and $j$ cannot be connected, then there is no $k$ such that $(A^k)_{i,j}\geq 1$, as denoted by $d(i,j)=\infty$. If $A$ is irreducible, then for any vertices $i,j\in V$ can be connected in $G$, that is, there is a $k\geq 1$ such that $(A^k)_{i,j}\geq 1$. The diameter $D(G)$ of $G$ is the maximum of distances of all $i,j\in V$, that is, 
\begin{equation}
D(G)=\max \left\{ d(i,j) : i,j \in V  \right\}.
\end{equation}

We now introduce the following notion of the uniform connectedness property of $\{ {\bf T}_m  \}_{m=1}^{\infty}$.

\begin{definition}\label{def3.12}
	For a basic set of admissible patterns $\mathcal{B}$, ${\bf T}_m(\mathcal{B})$ is called \emph{uniformly connected} if ${\bf T}_m(\mathcal{B})$ is irreducible for any $m\geq 1$, and there is a positive integer $K$ such that for any $m \geq 1$ and any $1 \leq i,j \leq r^m$, there exists a $1\leq k \leq K$ such that
	\begin{equation}\label{4.27}
	({\bf T}_m^k)_{i,j} \geq 1.
	\end{equation}
\end{definition}

Uniform connectedness of $\{ {\bf T}_m \}_{m=1}^{\infty}$ is equivalent to there is a finite bound $K$ of the diameters of associated graphs $G({\bf T}_m)$, that is,
\begin{equation}
\max \left\{ D(G({\bf T}_m)) : m \in \mathbb{N} \right\}  \leq K.
\end{equation}

\begin{theorem}
	If $\{ {\bf T}_m  \}_{m=1}^{\infty}$ is irreducible and uniformly connected then $\{ {\bf T}_m  \}_{m=1}^{\infty}$ is uniformly dominated by $\{ \rho({\bf T}_m ) \}_{m=1}^{\infty}$.
\end{theorem}

\begin{proof}
	Let $K\geq 1$ such that (\ref{4.27}) holds. Let $v_m=(v_{m,1}, ..., v_{m,N(m)})$ be the associated positive eigenvectors and $\rho_m=\rho({\bf T}_m)$ be the maximum positive eigenvalue of ${\bf T}_m$. Then Lemma \ref{lemma 4.10} implies
	\begin{equation}\label{4.28}
	\frac{\max_i v_{m,i}}{\min_i v_{m,i} } \leq \rho_m^K.
	\end{equation}
	By (\ref{4.22}),
	\begin{equation}\label{4.29}
	\left| {\bf T}_m^k  \right| \leq \rho_m^K r^m \rho_m^k
	\end{equation}
	for all $m\geq 1$ and $k \geq 1$.
	
	Since $\rho_m \leq \left| \mathcal{S} \right|^m$, hence
	\begin{equation}\label{4.30}
	\left| {\bf T}_m^k  \right| \leq r^{Km} r^m \rho_m^k=\left( r^{K+1} \right)^m \rho_m^k.
	\end{equation}
	(\ref{4.17}) holds with $C=r^{m(K+1)}$. Hence $\{ {\bf T}_m  \}_{m=1}^{\infty}$ is uniformly dominated by $\{ \rho({\bf T}_m ) \}_{m=1}^{\infty}$. The proof is complete.
\end{proof}

Therefore we have the following result. 
\begin{theorem}\label{thm 4.14}
	If $\{ {\bf T}_m(\mathcal{B})  \}_{m=1}^{\infty}$ is irreducible and uniformly connected then 
	\begin{equation*}
	h(\mathcal{B})=h_p(\mathcal{B})=h_*(\mathcal{B}).
	\end{equation*}
\end{theorem}

In the case of ${\bf T}_m$ being reducible, then Theorem \ref{theorem 4.2} (ii) implies there is a maximum irreducible submatrix $\overline{\bf T}_m=[\overline{t}_{m;\alpha, \beta}]_{\overline{\mathcal{I}}_m \times \overline{\mathcal{I}}_m}$ where $\overline{t}_{m;\alpha, \beta} \leq t_{m;\alpha, \beta}$ and either $\overline{t}_{m;\alpha, \beta} = t_{m;\alpha, \beta}>0$ or $\overline{t}_{m;\alpha, \beta}=0$ and $\overline{\mathcal{I}}_m \subseteq [1,r^m]$ is the set of indices such that for any indices pair $\alpha$ and $\beta$ in $\overline{\mathcal{I}}_m$ then there is $k\geq 1$ such that 
\begin{equation}\label{4.31}
(\overline{\bf T}_m^k )_{\alpha, \beta} \geq 1.
\end{equation}
Since $\overline{\bf T}_m$ is maximum, 
\begin{equation}\label{4.32}
\rho(\overline{\bf T}_m)=\rho({\bf T}_m).
\end{equation}
We have the following definition.

\begin{definition}\label{lemma 4.15}
	For any basic set of admissible patterns $\mathcal{B}\subseteq \mathcal{S}_r$, $ {\bf T}_m^k(\mathcal{B})$ is called uniformly dominated by $\overline{\bf T}_m^k(\mathcal{B})$ if there is a positive function $c_0(m,k)$ such that 
	\begin{equation}\label{4.33}
	\left| {\bf T}_m^k(\mathcal{B}) \right| \leq c_0(m,k)\left| \overline{\bf T}_m^k(\mathcal{B}) \right|
	\end{equation}
	with
	\begin{equation}\label{4.33-1}
	\lim_{k \rightarrow \infty} \frac{1}{mk}\log c_0(m,k)=0.
	\end{equation}
\end{definition} 

To illustrate Definition \ref{lemma 4.15}, we introduce the following example.
\begin{example}
	Consider
	\begin{equation*}
	{\bf V}_2=\left[
	\begin{matrix}
	1&0&0&1\\
	1&1&1&1\\
	1&1&1&1\\
	1&0&0&1
	\end{matrix}\right]
	\end{equation*}
	which is a reducible matrix. Then by permutation matrix
	\begin{equation*}
	P=\left[
	\begin{matrix}
	1&0&0&0\\
	0&0&0&1\\
	0&0&1&0\\
	0&1&0&0
	\end{matrix}\right],
	\end{equation*}
	${\bf V}_2$ can be transformed into its normal form of reducible matrices
	\begin{equation*}
	\overline{\bf V}_2=\left[
	\begin{matrix}
	1&1&0&0\\
	1&1&0&0\\
	1&1&1&1\\
	1&1&1&1
	\end{matrix}\right]=
	\left[
	\begin{matrix}
	1&0\\
	1&1
	\end{matrix}\right]\otimes
	\left[
	\begin{matrix}
	1&1\\
	1&1
	\end{matrix}\right].
	\end{equation*}	
	By using the recursive formula, it can be shown that the normal form $\overline{\bf V}_3$ of matrix ${\bf V}_3$ is
	\begin{equation*}
	\overline{\bf V}_3=\left[
	\begin{matrix}
	1&0\\
	1&1
	\end{matrix}\right]\otimes
	\left[
	\begin{matrix}
	1&0\\
	1&1
	\end{matrix}\right]\otimes
	\left[
	\begin{matrix}
	1&1\\
	1&1
	\end{matrix}\right].
	\end{equation*}		
	For general $m\geq 3$,
	\begin{equation*}
	\overline{\bf V}_m=\otimes^{m-1}
	\left[
	\begin{matrix}
	1&0\\
	1&1
	\end{matrix}\right]\otimes
	\left[
	\begin{matrix}
	1&1\\
	1&1
	\end{matrix}\right].
	\end{equation*}	
	Therefore,
	\begin{equation*}
	\left|{\bf V}_m^k  \right|=(k+2)^{m-1}2^{k+1}.
	\end{equation*}
	Hence
	\begin{equation*}
	\left|{\bf T}_m^k  \right|\leq c_0(m,k)2^{k+1}
	\end{equation*}
	with $c_0(m,k)=(k+2)^{m-1}$, which satisfies (\ref{4.33-1}). 
	
	Note that
	\begin{equation*}
	\rho({\bf T}_m)=\rho({\bf V}_m)=2
	\end{equation*} 
	for all $m\geq 2$. Hence,
	\begin{equation*}
	h(\mathcal{B})=h_p(\mathcal{B})=0.
	\end{equation*}
\end{example}

It may happen that ${\bf T}_m(\mathcal{B})$ has zero rows or zero columns. In these case, we can reduce ${\bf T}_m(\mathcal{B})$ by deleting its zero rows and columns. When matrix $A$ has zero rows or zero columns, we introduce the following notation, see \cite{BLL2009 Mixing property and entropy conjugacy of Z2 subshift of finite type: a case study}.
\begin{definition}
	A non-negative matrix $A$ is called \emph{weakly primitive} if there exists $K\geq 1$ such that each entry of $A^k$ is positive except in positions of $A$ where a zero row or zero column is present for all $k\geq K$. That is to say after deleting the zero row or zero column in $A$, the remaining matrix $\overline{A}$ is primitive. Similarly, $A$ is called \emph{weakly irreducible} if the remaining matrix $\overline{A}$ is irreducible after deleting zero rows and zero columns from $A$.
\end{definition}
Definition \ref{def3.12} can be extended as follows. 
\begin{definition}
The maximum irreducible component $\{ \overline{\bf T}_m(\mathcal{B})  \}_{m=1}^{\infty}$ of $\{ {\bf T}_m(\mathcal{B}) \}_{m=1}^\infty$ is called uniformly connected if there is a positive constant $K\geq 1$ such that for any $m\geq 1$ and any indices pair $i,j \in \overline{\mathcal{I}}_m$, there is $1 \leq k \leq K$ such that 
\begin{equation}\label{4.34}
(\overline{\bf T}_m^k )_{i,j} \geq 1.
\end{equation}
\end{definition}
Therefore, we have the following result.

\begin{theorem}\label{thm 4.17}
	If $\{ {\bf T}_m(\mathcal{B})  \}_{m=1}^{\infty}$ is uniformly dominated by the maximum irreducible submatrices $\{ \overline{\bf T}_m(\mathcal{B})  \}_{m=1}^{\infty}$ and $\{ \overline{\bf T}_m(\mathcal{B})  \}_{m=1}^{\infty}$ is uniformly connected, then $\{ {\bf T}_m(\mathcal{B})  \}_{m=1}^{\infty}$ is uniformly dominated by $\{ \rho ({\bf T}_m) \}_{m=1}^{\infty}$. In theses cases, 
	\begin{equation}
	h(\mathcal{B})=h_p(\mathcal{B})=h_*(\mathcal{B}).
	\end{equation}
\end{theorem}

We can provide some examples which have been studied widely as follows.

\begin{example}\label{example 3.20}
	\item[(a)] Golden-Mean shift (GM):	
	\begin{equation*}
	{\bf H}_2=\left[ \begin{matrix}
	1&1&1&0\\
	1&0&1&0\\
	1&1&0&0\\
	0&0&0&0
	\end{matrix}  \right].
	\end{equation*}
	
	\item[(b)] Hard-Hexagon model \cite{Joyce1988 On the Hard-Hexagon}:
	\begin{equation*}
	{\bf H}_2=\left[ \begin{matrix}
	1&1&1&0\\
	1&0&1&0\\
	1&0&0&0\\
	0&0&0&0
	\end{matrix}  \right].
	\end{equation*}
	
	\item[(c)]Strict Golden-Mean shift (SGM):
	\begin{equation*}
	{\bf H}_2=\left[ \begin{matrix}
	1&1&1&0\\
	1&0&0&0\\
	1&0&0&0\\
	0&0&0&0
	\end{matrix}  \right].
	\end{equation*}
	All of their ${\bf H}_n$, ${\bf V}_n$ and ${\bf T}_m$ can be proved to be weakly primitive and $\{\overline{\bf T}_m \}_{m=1}^\infty$ are uniformly connected. Therefore, Theorem \ref{thm 4.17} is applicable to them.
\end{example}

It is worth to mention that the concept of uniform connectedness of $\left\{  {\bf T}_m \left( \mathcal{B}\right) \right\}_{m=1}^{\infty}$ can also be interpreted in terms of the mixing properties of a two-dimensional symbolic dynamic system. Indeed, it is a periodic version of block gluing which has been studied in \cite{BHLL2021 Verification of mixing properties in two-dimensional shifts of finite type,Boyle2010 Multidimensional sofic shifts}. We can now introduce the concept of horizontal-periodic block gluing as follows.

For $m\geq 1$ and a pair of two horizontal $m$-periodic pattern $\overline{U}_m\in \Sigma_{\infty\times k_1}\left( \mathcal{B} \right)$ and $\overline{V}_m\in \Sigma_{\infty\times k_2}\left( \mathcal{B} \right)$, $\left(\overline{U}_m,\overline{V}_m  \right)$ can be glued together with vertical distance $k\geq 0$ if there exists a $m$-periodic patterns $\overline{W}_m\in \Sigma_{\infty\times (k_1+k+k_2)}\left( \mathcal{B} \right)$ such that
\begin{equation}\label{3.48}
\left(\overline{W}_m\right)_{\alpha,\beta}=\left(\overline{U}_m\right)_{\alpha,\beta},\mbox{ for }1-k_1\leq \beta \leq 0
\end{equation}
and
\begin{equation}\label{3.49}
\left(\overline{W}_m\right)_{\alpha,\beta}=\left(\overline{V}_m\right)_{\alpha,\beta},\mbox{ for }k \leq \beta \leq k+k_2-1.
\end{equation}

We can now give the definition of horizontal-periodic block gluing.
\begin{definition}\label{def3.21}
	The shift space $\Sigma \left( \mathcal{B}\right)$ is called \emph{horizontal-periodic block gluing} if there exists an integer $K\geq 1$ such that for any two horizontal $m$-periodic patterns $\overline{U}_m\in \Sigma_{\infty\times k_1}\left( \mathcal{B} \right)$ and $\overline{V}_m\in \Sigma_{\infty\times k_2}\left( \mathcal{B} \right)$, $\left(\overline{U}_m,\overline{V}_m  \right)$ can be glued together with vertical distance $k$ for any $k\geq K$ for all $m\geq 1$.
\end{definition}
Now we can prove the following theorem which is concerning the uniform connectness and horizontal-periodic block gluing.
\begin{theorem}\label{theorem 3.22}
	If $\left\{  {\bf T}_m \left( \mathcal{B}\right) \right\}_{m=1}^{\infty}$ is irreducible and $\Sigma \left(\mathcal{B}\right)$ is horizontal-periodic block gluing, then $\left\{  {\bf T}_m \left( \mathcal{B}\right) \right\}_{m=1}^{\infty}$ is uniformly connected. Furthermore, if $\left\{  {\bf T}_m \left( \mathcal{B}\right) \right\}_{m=1}^{\infty}$ is uniformly connected and for any $m\geq 1$, there exists an index $i_m$ such that 
	\begin{equation}\label{3.50}
	\left( {\bf T}_m\left( \mathcal{B}\right)\right)_{i_m,i_m}=1,
	\end{equation}
	then $\Sigma\left(\mathcal{B}\right)$ is horizontal-periodic block gluing.
\end{theorem}

\begin{proof}
	Denoted by the indices $\left(\overline{U}_m\right)_{\alpha,0}, 0\leq \alpha \leq m-1$ and $\left(\overline{V}_m\right)_{\alpha,0}, 0\leq \alpha \leq m-1$, by $i$ and $j$, respectively. Then (\ref{3.48}) and (\ref{3.49}) is equivalent to 
	\begin{equation}
		\left( {\bf T}^k_m\left( \mathcal{B}\right)\right)_{i,j}\geq 1.
	\end{equation}
	If $\Sigma\left( \mathcal{B}\right)$ is horizontal-periodic block gluing, then for any $i,j$ in $\mathcal{I}_m \left( \Sigma \left( \mathcal{B}\right)\right)$
	\begin{equation*}
	\left( {\bf T}^K_m\left( \mathcal{B}\right)\right)_{i,j}\geq 1.
	\end{equation*}
	Hence
	\begin{equation}
	\left( {\bf T}^k_m\left( \mathcal{B}\right)\right)_{i,j}\geq 1
	\end{equation}
	for some $1\leq k \leq K$. Therefore, $\left\{  {\bf T}_m \left( \mathcal{B}\right) \right\}_{m=1}^{\infty}$ is uniformly connected.
	
	On the other hand, if $\left\{  {\bf T}_m \left( \mathcal{B}\right) \right\}_{m=1}^{\infty}$ is uniformly connected with finite bound $K$ of diameters of associated graphs $G_m\left( \mathcal{B}\right)$ and satisfying condition (\ref{3.50}), then for $m\geq 1$ and any index pair $1\leq i,j \leq r^m$, if $k\geq 2K+1$ then
	\begin{equation*}
	\left( {\bf T}^{k_1}_m\left( \mathcal{B}\right)\right)_{i,i_m}\geq 1 \mbox{ and }\left( {\bf T}^{k_2}_m\left( \mathcal{B}\right)\right)_{i_m,j}\geq 1
	\end{equation*}
	with $1\leq k_1,k_2 \leq K$, and 
	\begin{equation*}
		\left( {\bf T}^{2K+1-(k_1+k_2)}_m\left( \mathcal{B}\right)\right)_{i_m,i_m}\geq 1
	\end{equation*}
	implies
	\begin{equation*}
		\left( {\bf T}^k_m\left( \mathcal{B}\right)\right)_{i,j}\geq 1.
	\end{equation*}
	
	Hence $\Sigma\left( \mathcal{B}\right)$ is horizontal-periodic block gluing. The proof is complete.
\end{proof}

By combining with Theorems \ref{thm 4.14}, \ref{thm 4.17} and \ref{theorem 3.22}, we have the following result.

\begin{theorem}
	If $\left\{  {\bf T}_m \left( \mathcal{B}\right) \right\}_{m=1}^{\infty}$ is weakly irreducible and $\Sigma\left( \mathcal{B}\right)$ is horizontal-periodic block gluing, then
	\begin{equation*}
	h\left(\mathcal{B}\right)=	h_p\left(\mathcal{B}\right)=	h_*\left(\mathcal{B}\right).
	\end{equation*} 
\end{theorem}

\section{Entropy studied by skew-coordinated system}\label{section4}
In this section, we use a skew-coordinated system to study the entropies. We first recall some properties of skew-coordinated systems \cite{BHLL2013 Zeta functions for two-dimensional shifts of finite type,Lind1998 A zeta function for,Mac Duffie1956 The theory of matrices}. For a skew-coordinated system $\gamma \in GL_2(\mathbb{Z})$, $GL_2(\mathbb{Z})$ is the modular group
 \begin{equation*}
	\begin{aligned}
	 GL_2(\mathbb{Z})=\left\{ \left[  \begin{matrix}
	 	a&b \\
	 	c&d
	 \end{matrix}  \right] : a,b,c,d\in \mathbb{Z} \mbox{ and } \left|  ad-bc  \right|=1    \right\},
	\end{aligned}
\end{equation*}	
and its subgroup
 \begin{equation*}
	\begin{aligned}
		SL_2(\mathbb{Z})=\left\{ \left[  \begin{matrix}
			a&b \\
			c&d
		\end{matrix}  \right]    :  a,b,c,d\in \mathbb{Z} \mbox{ and }   ad-bc  =1       \right\}.
	\end{aligned}
\end{equation*}	

Given $\gamma=\left[  \begin{matrix}
a&b \\
c&d
\end{matrix}  \right] \in GL_2(\mathbb{Z})$, $\mathbb{Z}^2=\{ (ra+sc, rb+sd) : r,s \in \mathbb{Z} \}$ holds. Therefore $\gamma$ is a unimodular transformation on $\mathbb{Z}^2$ and induces a skew-coordinated system on $\mathbb{Z}^2$. Indeed, the unit lattice points in $\gamma$-coordinates are
\[
(1,0)_{\gamma}=(a,b) \mbox{ and } (0,1)_{\gamma}=(c,d),
\]
and the unit vectors are 
\[
\vec{\gamma}_1=(1,0)_{\gamma}^t=(a,b)^t\mbox{ and }\vec{\gamma}_2=(0,1)_{\gamma}^t=(c,d)^t.
\]
The height $h(\gamma)$ of $\gamma$ is defined by
\[
h=h(\gamma)=|a|+|b|,
\]
and the width of $\gamma$ is
\[
w=w(\gamma)=|c|+|d|.
\]
$\gamma_0=\left[  \begin{matrix}
1&0\\
0&1
\end{matrix}  \right]\in GL_2(\mathbb{Z})$ is the standard rectangular system. The smallest parallelogram lattices in the $\gamma$-coordinate that contain exactly one unit square lattice in $\gamma_0$-coordinates are determined as follows, see \cite{BHLL2013 Zeta functions for two-dimensional shifts of finite type}.
\begin{proposition}\label{theorem unit square}
	For any $\gamma=\left[  \begin{matrix}
	a&b\\
	c&d
	\end{matrix}  \right]\in GL_2(\mathbb{Z})$, there exists exactly one unit square lattice in $\gamma_0$-coordinates in the parallelogram lattices with vertices $(0,0)_{\gamma}$, $(w,0)_{\gamma}$, $(0,h)_\gamma$ and $(w,h)_\gamma$. The unit square lattice has either vertices $(0,h)_\gamma$ and $(w,0)_\gamma$ or $(0,0)_\gamma$ and $(w,h)_\gamma$.
\end{proposition}
After Proposition \ref{theorem unit square}, the \emph{cylindrical transition matrices} ${\bf T}_{\gamma,n}(\mathcal{B})$ which determine the $\vec{\gamma}_1$-periodic with period $n$ and width $2$ in $\vec{\gamma}_2$-direction can be obtained.

Finally, as in $\gamma_0$-coordinates, the rotational matrix ${\bf R}_{\gamma,n}$ for $\gamma$-coordinates can be introduced. Then we can have same results as Porposition \ref{gamma n l o k } for $\gamma$-coordinates.
\begin{proposition}\label{gamma_r n l o k }
	For any $\gamma=\left[  \begin{matrix}
	a&b\\
	c&d
	\end{matrix}  \right]\in GL_2(\mathbb{Z})$, let $M,K\geq 1$ and $0\leq \L \leq M-1$,
	\begin{equation}
	\begin{aligned}
	\Gamma\left(  \left[  \begin{matrix}
	M&L \\
	0&K
	\end{matrix}  \right]_{\gamma} \right)=tr\left( {\bf T}_{\gamma,M}^K{\bf R}_{\gamma,M}^L  \right).
	\end{aligned}
	\end{equation}		
\end{proposition}
 
It is worth investigating the entropies computed by various skew-coordinated systems. Therefore, we will study the transformations of Hermite normal forms among the different skew-coordinated systems.

We first recall the method for transforming a $2\times 2$ integer matrix into its Hermite normal form, see \cite{Mac Duffie1956 The theory of matrices}. We denote by $A=\left[  \begin{matrix}
a_{11}&a_{12}\\
a_{21}&a_{22}
\end{matrix}\right]$ an integer matrix with $\det A\neq 0$ and $a_{21}\neq 0$. If $a_{21}=0$, then $A$ is already in a normal form.

Let 
\begin{equation}\label{k=gcd}
k=\gcd(a_{21},a_{22}),
\end{equation}
be the greatest common divisor of $a_{21}$ and $a_{22}$ with 
\begin{equation}\label{k=b1a+b2a}
k=b_1a_{21}+b_2a_{22}.
\end{equation}
Then $k\neq 0$. When $a_{22}=0$, then $k=a_{21}$ and $b_1=1$. Let
\begin{equation}\label{U}
\overline{U}=\left[  \begin{matrix}
\frac{a_{22}}{k}& b_1\\
-\frac{a_{21}}{k}& b_2\\
\end{matrix}\right].
\end{equation}
Then by (\ref{k=b1a+b2a}), $\det\overline{U}=1$, i.e., $\overline{U}$ is a \emph{unitary matrix}.
Then 
\begin{equation}\label{AU}
A\overline{U}=\left[\begin{matrix}
m&\ell\\
0&k
\end{matrix}\right]
\end{equation}
with
\begin{equation}\label{m=}
m=\frac{1}{k}\det A
\end{equation}
and 
\begin{equation}\label{ell=b1a+b2a}
\ell=b_1a_{11}+b_2a_{12}.
\end{equation}
Hence, 
\begin{equation}\label{AZ=}
A\mathbb{Z}^2=\left[\begin{matrix}
m&\ell\\
0&k
\end{matrix}\right]\mathbb{Z}^2.
\end{equation}
Therefore, $\left[\begin{matrix}
m&\ell\\
0&k
\end{matrix}\right]$ is the Hermite normal form of $A$.

Note that, two integer $2\times 2$ matrices $A$ and $A'$ are call \emph{equivalent} and are denoted by $A'\cong A$ if 
\begin{equation}
A'\mathbb{Z}^2=A\mathbb{Z}^2,
\end{equation} 
i.e., they determine the same sublattices in $\mathbb{Z}^2$. Hence $A$ and its normal form $\left[\begin{matrix}
m&\ell\\
0&k
\end{matrix}\right]$ are equivalent in (\ref{AZ=}) which determine the same sublattices of $\mathbb{Z}^2$. After introducing the procedure to transform the $2\times 2$ integer matrix to its normal form, we are going to study the transformation of normal form $\left[\begin{matrix}
M&L\\
0&K
\end{matrix}\right]_{\gamma}$ in $\gamma$-coordinates to its normal form $\left[\begin{matrix}
m&\ell\\
0&k
\end{matrix}\right]_{\gamma_0}$ in $\gamma_0$-coordinates. 

Firstly, $M_{\gamma}=\left[\begin{matrix}
M&L\\
0&K
\end{matrix}\right]_{\gamma}$ with respect to $\gamma$-coordinates can be rewritten as in $\gamma_0$-coordinates by 
\begin{equation}\label{Mgamma to gamma0}
	\gamma^tM_\gamma=\left[\begin{matrix}
		a&c\\
		b&d
	\end{matrix}\right]\left[\begin{matrix}
		M&L\\
		0&K
	\end{matrix}\right]=\left[\begin{matrix}
		aM&aL+cK\\
		bM&bL+dK
	\end{matrix}\right]_{\gamma_0}
\end{equation}
It remains to transform (\ref{Mgamma to gamma0}) to its normal form. Indeed, assume $b\neq 0$, let 
\begin{align}
	k&=\gcd(bM,bL+dK)\label{Mgamma2}\\
	&=b_1(bM)+b_2(bL+dK),\label{Mgamma3}
\end{align}
\begin{equation}\label{Mgamma4}
	\Delta=\det \gamma,
\end{equation}
and
\begin{equation}\label{Mgamma5}
	\overline{U}_{\gamma}=\left[\begin{matrix}
		\frac{\Delta}{k}(bL+dK)&b_1\\
		-\frac{\Delta}{k}(bM)&b_2
	\end{matrix}\right].
\end{equation}
Then, it is straightforward to verify that 
\begin{equation}\label{Mgamma6}
	\gamma^tM_\gamma \overline{U}_{\gamma}=\left[\begin{matrix}
		m&\ell\\
		0&k
	\end{matrix}\right]_{\gamma_0}
\end{equation}
with
\begin{equation}\label{Mgamma7}
	m=\frac{MK}{k}
\end{equation}
and
\begin{equation}\label{Mgamma8}
	\ell=b_1(aM)+b_2(aL+cK).
\end{equation}
That is,
\begin{equation}\label{Mgamma9}
	\left[\begin{matrix}
		m&\ell\\
		0&k
	\end{matrix}\right]_{\gamma_0}=\left[\begin{matrix}
		\frac{MK}{k}&b_1(aM)+b_2(aL+cK)\\
		0&k
	\end{matrix}\right]_{\gamma_0}.
\end{equation}

Therefore, we have the following theorem.
\begin{theorem}\label{theorem1}
	Given $\gamma=\left[\begin{matrix}
	a&b\\
	c&d
	\end{matrix}\right]\in GL_2(\mathbb{Z})$, $\Delta=\det\gamma$,  with $b\neq 0$. Let $k=\gcd(bM,bL+dK)$ and $k=b_1(bM)+b_2(bL+dK)$. Let 
	\begin{equation}\label{thm1}
		bM=m'k
	\end{equation}
	and
	\begin{equation}\label{thm2}
		bL+dK=\ell'k
	\end{equation}
	with
	\begin{equation}\label{thm3}
		b_1m'+b_2\ell'=1.
	\end{equation}
	Then
	\begin{equation}\label{thm4}
		\left[\begin{matrix}
			\frac{m'k}{b}& \frac{-dK+\ell'k}{b}    \\
			0&K
		\end{matrix}\right]_{\gamma}\cong\left[\begin{matrix}
			\frac{m'K}{b}&\frac{ak-\Delta b_2K}{b}\\
			0&k
		\end{matrix}\right]_{\gamma_0}.
	\end{equation}
\end{theorem}	

\begin{proof}
By (\ref{thm1}) and (\ref{thm2}), $M=\frac{m'k}{b}$ and $L=\frac{\ell'k-dK}{b}$. On the other hand, by (\ref{Mgamma7}) and (\ref{Mgamma8}), $m=\frac{MK}{b}=\frac{m'K}{b}$ and	\begin{equation*}
	\begin{aligned}
		\ell&=b_1(aM)+b_2(aL+cK)\\
		&=\frac{1}{b}\left\{   a(b_1bM+b_2bL+b_2dK)+(bc-ad)b_2K   \right\}\\
		&=\frac{1}{b}(ak-\Delta b_2 K).
	\end{aligned}
\end{equation*}
The proof is complete.	
\end{proof}
Also, (\ref{thm4}) in Theorem \ref{theorem1} can be stated as the following theorem.
\begin{theorem}\label{theorem1-1}
	Given $\gamma=\left[\begin{matrix}
	a&b\\
	c&d
	\end{matrix}\right]\in GL_2(\mathbb{Z})$, $\Delta=\det\gamma$, with $b\neq 0$. 
	
	Let 
	\begin{equation}\label{thm1-1}
		\ell'k=\mu_0 b+\nu_0,
	\end{equation}
	and
	\begin{equation}\label{thm2-1}
		K=\mu b+\nu,
	\end{equation}
	where
	\begin{equation}\label{thm3-1}
		0\leq \nu_0 \leq b-1\mbox{ and }0 \leq \nu \leq b-1.
	\end{equation}
	Then
	\begin{equation}\label{thm4-1}
		\left[\begin{matrix}
			\frac{m'k}{b}& \mu_0-d\mu+\frac{\nu_0+d\nu}{b}    \\
			0&\mu b+\nu
		\end{matrix}\right]_{\gamma}\cong\left[\begin{matrix}
			m'\mu+\frac{m'\nu}{b}&-\Delta b_2\mu+\frac{ak-\Delta b_2\nu}{b}\\
			0&k
		\end{matrix}\right]_{\gamma_0}.
	\end{equation}
	Furthermore,
	\begin{equation}\label{thm5-1}
		\mbox{if } b | m'k\mbox{ then } b|m'\nu,
	\end{equation}
	and
	\begin{equation}\label{thm6-1}
		\mbox{if } b | (\nu_0-d\nu)\mbox{ then } b|(ak-\Delta b_2\nu).
	\end{equation}	
\end{theorem}

\begin{proof}
By (\ref{thm1-1}) and (\ref{thm2-1}), $\frac{\ell'k-dK}{b}=\mu_0-d\mu +\frac{\nu_0-d\nu}{b}$. Given $0\leq\nu_0 \leq b-1$, since $\gcd(b,d)=1$, then there exist $m,n\in\mathbb{Z}$ such that $b\nu_0 m+d\nu_0 n=\nu_0$. This implies $b(\nu_0 m-dk)+d(\nu_0 n-bk)=\nu_0$ for all $k\in \mathbb{Z}$, which gives a unique $0\leq \nu \leq b-1$ such that $b|\nu_0-d\nu$. Furthermore, $ak-\Delta b_2 K=-b\Delta b_2\mu +ak-\Delta b_2\nu$. Hence $\ell=-\Delta b_2\mu+\frac{ak-\Delta b_2\nu}{b}$. Finally, (\ref{thm5-1}) and (\ref{thm6-1}) can be verified directly and the details are omitted here.	
\end{proof}

Theorems \ref{theorem1} and \ref{theorem1-1} can also be generalized to any two $\gamma'=\left[\begin{matrix}
a'&b'\\
c'&d'
\end{matrix}\right]$ and $\gamma=\left[\begin{matrix}
a&b\\
c&d
\end{matrix}\right]\in GL_2(\mathbb{Z})$ as in the following theorem.
\begin{theorem}\label{theorem1-2}
	Given $\gamma'=\left[\begin{matrix}
	a'&b'\\
	c'&d'
	\end{matrix}\right]$ and $\gamma=\left[\begin{matrix}
	a&b\\
	c&d
	\end{matrix}\right]\in GL_2(\mathbb{Z})$ with 
	\begin{equation}\label{1-2-1}
		a'b-ab'\neq 0,
	\end{equation}
	and
	\begin{equation}\label{1-2-2}
		\begin{aligned}
			b_1m'+b_2\ell'=1,~
			\Delta=\det \gamma \mbox{ and }\Delta'=\det \gamma'.	
		\end{aligned}
	\end{equation}
	Then
	\begin{equation}\label{1-2-3}
		\left[\begin{matrix}
			\frac{m'K'}{a'b-ab'}& \frac{\ell'K'-(a'd-b'c)K}{a'b-ab'}    \\
			0&K
		\end{matrix}\right]_{\gamma}\cong\left[\begin{matrix}
			\frac{m'K}{a'b-ab'}&\frac{K'(ad'-bc')}{a'b-ab'}-b_2K\frac{\Delta \Delta'}{a'b-ab'}\\
			0&K'
		\end{matrix}\right]_{\gamma'}.
	\end{equation}
\end{theorem}

\begin{proof}
For $\gamma,\gamma'\in GL_2(\mathbb{Z})$, if $	\left[ \begin{matrix}
M&L\\
0&K
\end{matrix}\right]_\gamma$ and $	\left[ \begin{matrix}
M'&L'\\
0&K'
\end{matrix}\right]_{\gamma'}$  are equivalent in $\gamma_0$-coordinates, then
\begin{equation}\label{proof1}
	\left[ \begin{matrix}
		a&c\\
		b&d
	\end{matrix}\right]\left[ \begin{matrix}
		M&L\\
		0&K
	\end{matrix}\right]U\cong\left[ \begin{matrix}
		a'&c'\\
		b'&d'
	\end{matrix}\right]\left[ \begin{matrix}
		M'&L'\\
		0&K'
	\end{matrix}\right]U',
\end{equation}
where $U$ and $U'$ are unitary matrices.

Then (\ref{proof1}) implies
\begin{equation}\label{proof2}
	\begin{aligned}
		\Delta'\left[ \begin{matrix}
			M'&L'\\
			0&K'
		\end{matrix}\right]&\cong\left[ \begin{matrix}
			d'&-c'\\
			-b'&a'
		\end{matrix}\right]\left[ \begin{matrix}
			aM&aL+cK\\
			bM&bL+dK
		\end{matrix}\right]U(U')^{-1}\\
		&\cong\left[ \begin{matrix}
			(ad'-bc')M&(ad'-bc')L+(cd'-c'd)K\\
			(a'b-ab')M&(a'b-ab')L+(ad'-b'c)K
		\end{matrix}\right]U(U')^{-1}.
	\end{aligned}
\end{equation}
Since $\Delta'=\pm 1$,
\begin{equation}\label{proof3}
	\begin{aligned}
		\Delta'\left[ \begin{matrix}
			M'&L'\\
			0&K'
		\end{matrix}\right]\cong\left[ \begin{matrix}
			M'&L'\\
			0&K'
		\end{matrix}\right]\cong\left[ \begin{matrix}
			\Delta'M'&L'\\
			0&K'
		\end{matrix}\right].
	\end{aligned}
\end{equation}
Combining (\ref{proof2}) and (\ref{proof3}), we obtain
\begin{equation}\label{proof4}
	\begin{aligned}
		\left[ \begin{matrix}
			M'&L'\\
			0&K'
		\end{matrix}\right]\cong\left[ \begin{matrix}
			(ad'-bc')M&(ad'-bc')L+(cd'-c'd)K\\
			(a'b-ab')M&(a'b-ab')L+(ad'-b'c)K
		\end{matrix}\right]U(U')^{-1}.
	\end{aligned}
\end{equation}

In the case $a'b-ab'\neq 0$. Since 
\begin{equation*}
	\begin{aligned}
		U(U')^{-1}=\left[ \begin{matrix}
			\frac{(a'b-ab')L+(a'd-b'c)K}{K'}&b_1\\
			\frac{(a'b-ab')M}{K'}&b_2
		\end{matrix}\right]
	\end{aligned}
\end{equation*}
where
\begin{equation*}
	\begin{aligned}
		K'&=\gcd((a'b-ab')M,(a'b-ab')L+(ad'-b'c)K)\\
		&=b_1(a'b-ab')M+b_2\left[ (a'b-ab')L+(ad'-b'c)K)\right]\\
		&=b_1m'K'+b_2 \ell' K'.
	\end{aligned}
\end{equation*}
Then $M=\frac{m'K'}{a'b-ab'}$, $L=\frac{\ell' K'-(a'd-b'c)K}{a'b-ab'}$, 
\begin{equation*}
	\begin{aligned}
		M'&=(ad'-bc')M\left[ \frac{(a'b-ab')L+(a'd-b'c)K}{K'}\right]-\frac{(a'b-ab')M}{K'} \left[(ad'-bc')L+(cd'-c'd)K \right]\\
		&=\frac{MK}{K'}\left[ (ad'-bc')(a'd-b'c)-(a'b-ab')(cd'-c'd)\right]\\
		&=\frac{MK}{K'}\Delta \Delta'\\
		&=\frac{m'K}{a'b-ab'}\Delta \Delta',
	\end{aligned}
\end{equation*}
and
\begin{equation*}
	\begin{aligned}
		L'&=b_1(ad'-bc')M+b_2 \left[ (ad'-bc')L+(cd'-c'd)K \right]\\
		&=b_1(ad'-bc')\frac{m'K'}{a'b-ab'}+b_2\left[ \frac{ad'-bc'}{a'b-ab'}[\ell' K'-(a'd-b'c)K]+\frac{(a'b-ab')(cd'-c'd)K}{a'b-ab'} \right]\\
		&=\frac{b_1(ad'-bc')m'K'+b_2 \ell' K' (ad'-bc')}{a'b-ab'}-b_2 K \frac{(ad'-bc')(a'd-b'c)-(a'b-ab')(cd'-c'd)}{a'b-ab'}\\
		&=\frac{K'(ad'-bc')}{a'b-ab'}-b_2 K \frac{\Delta \Delta'}{a'b-ab'}.
	\end{aligned}
\end{equation*}
Thus,
\begin{equation*}
	\begin{aligned}	
		\left[\begin{matrix}
			\frac{m'K'}{a'b-ab'}& \frac{\ell'K'-(a'd-b'c)K}{a'b-ab'}    \\
			0&K
		\end{matrix}\right]_{\gamma}&\cong\left[\begin{matrix}
			\Delta \Delta'\frac{m'K}{a'b-ab'}&\frac{K'(ad'-bc')}{a'b-ab'}-b_2K\frac{\Delta \Delta'}{a'b-ab'}\\
			0&K'
		\end{matrix}\right]_{\gamma'}\\
		&\cong\left[\begin{matrix}
			\frac{m'K}{a'b-ab'}&\frac{K'(ad'-bc')}{a'b-ab'}-b_2K\frac{\Delta \Delta'}{a'b-ab'}\\
			0&K'
		\end{matrix}\right]_{\gamma'},
	\end{aligned}
\end{equation*}
the last $\cong$ due to $\Delta\Delta'=\pm 1$ and (\ref{proof3}).
\end{proof}

\begin{remark}
If $a'b-ab'=0$, then $\frac{b'}{a'}=\frac{b}{a}$. Now, $\gcd(a',b')=1$ and $\gcd(a,b)=1$ imply either $a'=a$ and $b'=b$ or $a'=-a$ and $b'=-b$. Indeed, we have $M'=(ad'-bc')M$, $L'=(ad'-bc')L+(cd'-c'd)K$ and $K'=(a'b-ab')L+(ad'-b'c)K$.
\end{remark}

It is easy to see when $\gamma'=\gamma_0=\left[ \begin{matrix}
1&0\\
0&1
\end{matrix}\right]$, then (\ref{1-2-3}) in Theorem \ref{theorem1-2} is reduced to (\ref{thm4}) in Theorem \ref{theorem1}. If $\gamma'=\hat{\gamma_0}=\left[ \begin{matrix}
0&1\\
1&0
\end{matrix}\right]$ the conjugate system of $\gamma_0$, then we have
\begin{theorem}
	Given $\gamma=\left[ \begin{matrix}
	a&b\\
	c&d
	\end{matrix}\right]\in GL_2(\mathbb{Z})$ with $a\neq 0$, 
	\begin{equation*}
	\begin{aligned}		
	\left[ \begin{matrix}
	\frac{\hat{m}\hat{k}}{a}&\frac{\hat{\ell}\hat{k}+cK}{a}\\
	0&K
	\end{matrix}\right]_\gamma\cong\left[ \begin{matrix}
	\frac{\hat{m}K}{a}&\frac{-b\hat{k}+\Delta \hat{b_2} K}{a}\\
	0&\hat{k}
	\end{matrix}\right]_{\hat{\gamma_0}},
	\end{aligned}
	\end{equation*}
	where $\hat{b_1}\hat{m}+\hat{b_2}\hat{\ell}=1$. 
\end{theorem}

\begin{proof}
	The proof is easily from Theorem \ref{theorem1-2}.
\end{proof}

In particular, if $\gamma=\gamma_0=\left[\begin{matrix}
1&0\\
0&1
\end{matrix}\right]$ and $\gamma'=\hat{\gamma_0}=\left[\begin{matrix}
0&1\\
1&0
\end{matrix}\right]$. Then we have
\begin{theorem}
	\begin{equation*}
	\begin{aligned}		
	\left[ \begin{matrix}
	m'\hat{k}&\ell' \hat{k}\\
	0&k
	\end{matrix}\right]_{\gamma_0}\cong\left[ \begin{matrix}
	m' k&b_2 k\\
	0&\hat{k}
	\end{matrix}\right]_{\hat{\gamma_0}},
	\end{aligned}
	\end{equation*}
	where $b_1m'+b_2\ell'=1$.	
\end{theorem}

For completeness, a similar expression for $\hat{\gamma_0}$ as Theorem \ref{theorem1-1} can also be written in the following theorem. The proof is omitted.
\begin{theorem}
	Given $\gamma=\left[ \begin{matrix}
	a&b\\
	c&d
	\end{matrix}\right]\in GL_2(\mathbb{Z})$ with $a\neq 0$. Then
	\begin{equation}
	\begin{aligned}		
	\left[ \begin{matrix}
	\frac{m''\hat{k}}{a}&c\hat{\mu}+\hat{\mu_0}+\frac{\hat{\nu_0}+c\hat{\nu}}{a}\\
	0&\hat{\mu}a+\hat{\nu}
	\end{matrix}\right]_{\gamma}\cong\left[ \begin{matrix}
	m''\hat{\mu}+\frac{m''\hat{\nu}}{a}&-b_2'' \hat{\mu}+\frac{b\hat{k}-\Delta b_2'' \hat{\nu}}{a} \\
	0&\hat{k}
	\end{matrix}\right]_{\hat{\gamma_0}},
	\end{aligned}
	\end{equation}
	where
	\begin{equation}
	b_1'' m''+b_2'' \ell''=1,~\ell''\hat{k}=\hat{\mu_0}a+\hat{\nu_0},
	\end{equation}	
	and
	\begin{equation}
	\hat{K}=\hat{\mu}a+\hat{\nu},~0\leq \hat{\nu_0} \leq a-1,~ 0\leq \hat{\nu} \leq a-1.
	\end{equation}
	Furthermore,
	\begin{equation}
	\mbox{if } a | m''\hat{k}\mbox{ then } a|m''\hat{\nu},
	\end{equation}
	and
	\begin{equation}
	\mbox{if } a | (\hat{\nu_0}+c\hat{\nu})\mbox{ then } a|(b\hat{k}-\Delta b_2'' \hat{\nu}).
	\end{equation}			
\end{theorem}

The entropy can be studied in skew-coordinated systems in any $\gamma=\left[ \begin{matrix}
a&b\\
c&d
\end{matrix}\right]\in GL_2(\mathbb{Z})$. Indeed, the parallelogram $\mathbb{Z}_{\gamma,m\times k}$ in $\gamma$-system is defined as $m$ units in $\vec{\gamma}_1$ direction and $k$ units in $\vec{\gamma}_2$ direction. Denoted by $\left|\Sigma_{\gamma;m\times k}(\mathcal{B}) \right|$ the total number of admissible patterns is determined by $\mathcal{B}$. Then the entropy $h_{\gamma}(\mathcal{B})$ computed in $\gamma$-coordinated system is defined by 
\begin{equation*}
h_{\gamma}(\mathcal{B})=\limsup_{(n,k)\rightarrow\infty}\frac{1}{nk}\log \left|\Sigma_{\gamma;m\times k}(\mathcal{B}) \right|.
\end{equation*}
By the results in \cite{HL2016 On spatial entropy},
\begin{equation*}
h_{\gamma}(\mathcal{B})=h(\mathcal{B})
\end{equation*}
for all basic set $\mathcal{B}$ and all $\gamma\in GL_2(\mathbb{Z})$.

The periodic entropy can also be defined in skew-coordinated systems. Indeed, for any $\gamma=\left[ \begin{matrix}
a&b\\
c&d
\end{matrix}\right]\in GL_2(\mathbb{Z})$, defining the periodic entropy $h_{\gamma,p}(\mathcal{B})$ of admissible basic set $\mathcal{B}\subseteq \mathcal{S}_r^{\mathbb{Z}_{2\times 2}}$ by 
\begin{equation*}
h_{\gamma,p}(\mathcal{B})=\limsup_{(M,K)\rightarrow\infty}\sup_{0\leq L \leq M-1}\frac{1}{MK} \log \Gamma \left(\left[ \begin{matrix}
	M&L\\
	0&K
\end{matrix}\right]_{\gamma} \right).
\end{equation*}
By the Theorem \ref{theorem1}, we can prove the following lemma.

\begin{lemma}\label{lemma 410}
	For any admissible basic set $\mathcal{B}$ and $\gamma\in GL_2(\mathbb{Z})$,
	\begin{equation*}
	h_{\gamma,p}(\mathcal{B})\geq h_*(\mathcal{B}).
	\end{equation*}
\end{lemma}
\begin{proof}
	We only treat the case $b\neq 0$. The case $b=0$ can also be studied analogously. By Theorem \ref{theorem1},
\begin{equation}
\left[\begin{matrix}
M&L    \\
0&K
\end{matrix}\right]_{\gamma}\cong\left[\begin{matrix}
m&\ell\\
0&k
\end{matrix}\right]_{\gamma_0},
\end{equation}	
where $M,L,K,m,\ell,k$ satisfy the relations in (\ref{thm4}). Then,
\begin{align*}
	h_{\gamma,p}(\mathcal{B})&=\limsup_{(M,K)\rightarrow\infty}\sup_{L }\frac{1}{MK} \log \Gamma_{\mathcal{B}} \left(\left[ \begin{matrix}
		M&L\\
		0&K
	\end{matrix}\right]_{\gamma}\right) \\
	&\geq \limsup_{K\to \infty} \left( \limsup_{M\to \infty} \sup_{L}\frac{1}{MK}\log \Gamma_{\mathcal{B}} \left(\left[\begin{matrix}
		M&L\\
		0&K
	\end{matrix}\right]_{\gamma}\right)\right)\\
	&=\limsup_{\frac{bm}{m'}\to \infty}
	\left( \limsup_{\frac{m'k}{b}\to \infty} \sup_{\ell=\frac{ak}{b}-\Delta b_2\frac{m}{m'}}\frac{1}{mk}\log \Gamma_{\mathcal{B}} \left(\left[\begin{matrix}
		m&\ell    \\
		0&k
	\end{matrix}\right]_{\gamma_0}\right)\right)\\
	&\geq \limsup_{m\to \infty}
\left( \limsup_{k\to \infty} \sup_{\ell=\frac{ak}{b}-\Delta b_2\frac{m}{m'}}\frac{1}{mk}\log \Gamma_{\mathcal{B}} \left(\left[\begin{matrix}
	m&\ell    \\
	0&k
\end{matrix}\right]_{\gamma_0}\right)\right).
\end{align*}	
By choosing $m'=1$, $b_1=1$, $b_2=0$ and $k=bm$, we have
\begin{align*}
	h_{\gamma,p}(\mathcal{B})&\geq \limsup_{m\to \infty}
	\frac{1}{m}\left( \limsup_{bm\to \infty} \sup_{am}\frac{1}{bm}\log \Gamma_{\mathcal{B}} \left(\left[\begin{matrix}
		m&am \\
		0&bm
	\end{matrix}\right]_{\gamma_0}\right)\right)\\
&=\limsup_{m\to \infty}
\frac{1}{m} \limsup_{bm\to \infty} \sup_{am}\frac{1}{bm}\log tr \left({\bf T}_m^{bm}(\mathcal{B}){\bf R}_m^{am}\right)\\
&= \limsup_{m\to \infty}
\frac{1}{m} \limsup_{bm\to \infty} \sup_{am}\frac{1}{bm}\log tr \left({\bf T}_m^{bm}(\mathcal{B})\right)\\
&= \limsup_{m\to \infty}
\frac{1}{m} \limsup_{bm\to \infty}\frac{1}{bm}\log tr \left({\bf T}_m^{bm}(\mathcal{B})\right)\\
&=\limsup_{m\to \infty} \frac{1}{m}\log \rho \left({\bf T}_m(\mathcal{B})\right)\\
&=h_*
\end{align*}	
The proof is complete.
\end{proof}

Hence we have the following theorem.
\begin{theorem}
	If $h(\mathcal{B})=h_*(\mathcal{B})$, then $h_{\gamma,p}(\mathcal{B})=h(\mathcal{B})$ for all $\gamma\in GL_2(\mathbb{Z})$. In particular, if $\{ {\bf T}_m(\mathcal{B})  \}_{m=1}^{\infty}$ is uniformly dominated by $\{ \rho ({\bf T}_m) \}_{m=1}^{\infty}$ which satisfies (\ref{4.17}) and (\ref{4.17-1}), then $h_{\gamma,p}(\mathcal{B})=h(\mathcal{B})$ for all $\gamma\in GL_2(\mathbb{Z})$.
\end{theorem}

\begin{proof}
	Since $h_{\gamma,p}(\mathcal{B})\leq h_{\gamma}(\mathcal{B})=h(\mathcal{B})$, the result follows from Lemma \ref{lemma 410}.
\end{proof}

Now, we are going to study the problems for which $\mathcal{B}$ can satisfy $h_\ell(\mathcal{B})=h_p(\mathcal{B})$. Recall from Proposition \ref{Rotation},
\begin{equation}\label{5.11}
R_{m;\alpha, \sigma^\ell(\alpha)}^{\ell}=1,
\end{equation}
or
\begin{equation}\label{5.12}
R_{m;\sigma^{-\ell}(\alpha), \alpha}^\ell=1
\end{equation}
for all $m\geq 1$ and $1\leq \alpha \leq r^m $.

Hence 
\begin{equation*}
({\bf T}_m {\bf R}_m^\ell)_{i,j}=\sum_{s} t_{m; i,s} R_{m; s,j}^\ell= t_{m; i, \sigma^{-\ell}(j)},
\end{equation*}
or
\begin{equation}\label{5.13}
({\bf T}_m {\bf R}_m^{\ell})_{i,j}= ({\bf T}_m)_{i, \sigma^{-\ell}(j)},
\end{equation}
and
\begin{equation}\label{5.14}
({\bf T}_m {\bf R}_m^{\ell})_{i,i}= ({\bf T}_m)_{i, \sigma^{-\ell}(i)}.
\end{equation}
Therefore
\begin{equation}\label{5.15}
\Gamma \left[ \begin{matrix}
m&\ell\\
0&k
\end{matrix}\right]= \sum_{i=1}^{r^m} ({\bf T}_m^k)_{i, \sigma^{-\ell}(i)}=\sum_{i_1=1}^{r^m}  \sum_{i_2,...,i_k=1}^{r^m}  t_{m;i_1,i_2} \cdots t_{m; i_k,\sigma^{-\ell}(i_1)}.
\end{equation}

We have the following results. 
\begin{theorem}\label{thm 5.4}
	If ${\bf T}_m(\mathcal{B})$ is irreducible for all $m\geq 1$, then 
	\begin{equation}\label{5.155}
	h_\ell(\mathcal{B})\geq h_*(\mathcal{B}).
	\end{equation}
	Furthermore, if ${\bf T}_m(\mathcal{B})$ is not irreducible and $\overline{\bf T}_m(\mathcal{B}) \leq {\bf T}_m(\mathcal{B}) $ is a maximum irreducible submatrix with index set $\overline{\mathcal{I}}_m$ and there is an index pair $i$ and $\sigma^{-\ell}(i)\in \overline{\mathcal{I}}_m$. Then (\ref{5.155}) holds.
\end{theorem}

\begin{proof}
	By Theorem \ref{thm 4.3}, if ${\bf T}_m$ is irreducible then for any $m\geq 1$ and any pair $i$ and $j$ there is $\beta_m=\beta_m(i,j)\geq 0$ and cycle $p_m \geq 1$ and $\alpha_m \geq 1$ such that
	\begin{equation}\label{5.16}
	c_m \rho({\bf T}_m)^{\alpha p_m} \leq ( {\bf T}_m^{\alpha p_m+\beta_m} )_{i,j}\leq d_m \rho({\bf T}_m)^{\alpha p_m} 
	\end{equation}
	for $\alpha \geq \alpha_m \geq 1$, where $0<c_m \leq d_m$.
	
	Hence (\ref{5.16}) implies 
	\begin{equation*}
	\begin{aligned}
	\limsup_{k\rightarrow \infty} \frac{1}{k}\log tr({\bf T}_m^k{\bf R}_m^\ell)&\geq 	\limsup_{\alpha\rightarrow \infty} \frac{1}{\alpha p_m + \beta_m}\log \sum_{i=1}^{r^m} ({\bf T}_m^{\alpha p_m +\beta_m})_{i,\sigma^{-\ell}(i)}  \\ &\geq 	\limsup_{\alpha\rightarrow \infty} \log \rho({\bf T}_m).
	\end{aligned}
	\end{equation*}
	Therefore, 
	\begin{equation*}
	\begin{aligned}
	h_\ell(\mathcal{B})&=\limsup_{(m,k)\rightarrow\infty} \frac{1}{mk}\log tr({\bf T}_m^k{\bf R}_m^\ell)\\
	&\geq \limsup_{m \rightarrow \infty} \frac{1}{m} \left( \limsup_{k\rightarrow \infty}   \frac{1}{k}\log tr({\bf T}_m^k{\bf R}_m^\ell)  \right)\\
	&\geq \limsup_{m \rightarrow \infty} \frac{1}{m} \left( \limsup_{k\rightarrow \infty}   \frac{1}{k}\log \sum_{i=1}^{r^m}({\bf T}_m^k)_{i,\sigma^{-\ell}(i)}  \right)\\
	&\geq  \limsup_{m \rightarrow \infty} \frac{1}{m}\log \rho({\bf T}_m). 
	\end{aligned}	
	\end{equation*}
	Next, if ${\bf T}_m$ is not irreducible, but $\overline{\bf T}_m \leq {\bf T}_m$ is a maximum irreducible submatrix with cycle $p_m \geq 1$, $i$ and $\sigma^{-\ell}(i)\in \overline{\mathcal{I}}_m$, the index set of $\overline{\bf T}_m$. Then, as (\ref{5.16}), the following inequality holds.
	\begin{equation}\label{5.17}
	c_m \rho(\overline{\bf T}_m)^{\alpha p_m} \leq ( {\bf T}_m^{\alpha p_m+\beta_m} )_{i,\sigma^{-\ell}(i)}\leq d_m \rho(\overline{\bf T}_m)^{\alpha p_m} 
	\end{equation}
	for all $\alpha\geq \alpha_m \geq 1$ and some $\beta_m\geq 0$. 
	
	Hence (\ref{5.15}) holds. The proof is completed.
\end{proof}

Combing with Theorems \ref{thm 5.4}, \ref{thm 4.9} and \ref{thm 4.14} we have the following results.

\begin{theorem}\label{thm 5.5}
	If $\{ {\bf T}_m (\mathcal{B})  \}_{m=1}^{\infty}$ is irreducible and uniformly dominated by $\{ \rho \left( {\bf T}_m (\mathcal{B}) \right)  \}_{m=1}^{\infty}$, then
	\begin{equation}\label{5.*}
	h_\ell(\mathcal{B})=h_p(\mathcal{B})=h(\mathcal{B})=h_*(\mathcal{B})
	\end{equation}
	for all integer $\ell$. Furthermore, if $\{ {\bf T}_m (\mathcal{B})  \}_{m=1}^{\infty}$ is reducible and $\{ \overline{\bf T}_m (\mathcal{B})  \}_{m=1}^{\infty}$ is a sequence of maximum irreducible submatrices with an indices pair $i$ and $\sigma^{-\ell}(i)\in \overline{\mathcal{I}}_m$, and $\{ \overline{\bf T}_m (\mathcal{B})  \}_{m=1}^{\infty}$ is uniformly dominated by $\{ \rho \left( {\bf T}_m (\mathcal{B}) \right)  \}_{m=1}^{\infty}$. Then (\ref{5.*}) holds.
\end{theorem} 
\begin{proof}
	Since
	\begin{equation*}
	h_\ell(\mathcal{B})\leq h_p(\mathcal{B}).
	\end{equation*} 
	Theorems \ref{4.9} and \ref{4.16} imply
	\begin{equation*}
	h_p(\mathcal{B})=h(\mathcal{B})=h_*(\mathcal{B}).
	\end{equation*}
	Now, Theorem \ref{thm 5.4} implies (\ref{5.*}). The proof is complete.
\end{proof}

By combining Theorems \ref{thm 5.5} and \ref{theorem 3.22}, we have the following result.
\begin{theorem}\label{thm4.14}
	If $\left\{  {\bf T}_m \left( \mathcal{B}\right) \right\}_{m=1}^{\infty}$ is irreducible or weakly irreducible and $\Sigma\left( \mathcal{B}\right)$ is horizontal-periodic block gluing, then
\begin{equation*}
h_{\ell}\left(\mathcal{B}\right)=	h_p\left(\mathcal{B}\right)=h\left(\mathcal{B}\right)=	h_*\left(\mathcal{B}\right),
\end{equation*} 
for all $\ell$.
\end{theorem}

\begin{example}
The GM, SGM and Hard-Hexagon model in Example \ref{example 3.20} satisfy the assumption of Theorem \ref{thm 4.14} and then (\ref{5.*}) holds.
\end{example}

The following theorem derives the transformation between $\left[ \begin{matrix}
M&L\\
0&K
\end{matrix}\right]_{\gamma_q}, q\geq 1$, $\left[ \begin{matrix}
m&\ell\\
0&k
\end{matrix}\right]_{\gamma_0}$, $\left[ \begin{matrix}
\hat{M}&\hat{L}\\
0&\hat{K}
\end{matrix}\right]_{\hat{\gamma}_q}$ and $\left[ \begin{matrix}
\hat{m}&\hat{\ell}\\
0&\hat{k}
\end{matrix}\right]_{\hat{\gamma_0}}$.

\begin{theorem}\label{thm 3.8}
	For any $q\geq 1$, 
	\begin{equation}\label{5.1}
	\left[ \begin{matrix}
	\frac{m' k}{q}&\frac{\ell' k-K}{q}\\
	0&K
	\end{matrix}\right]_{{\gamma}_q} \cong 
	\left[ \begin{matrix}
	\frac{m' K}{q}&\frac{ k-b_2 K}{q}\\
	0&k
	\end{matrix}\right]_{{\gamma}_0}
	\end{equation}
	where $b_1 m'+b_2 \ell'=1$. When $m'k=q$, then 
	\begin{equation}\label{5.2}
	\left[ \begin{matrix}
	1&0\\
	0&\alpha q+ \frac{\ell'}{m'} q
	\end{matrix}\right]_{{\gamma}_q} \cong 
	\left[ \begin{matrix}
	\alpha m' +\ell'&b_1-b_2 \alpha\\
	0& \frac{q}{m'}
	\end{matrix}\right]_{{\gamma}_0}
	\end{equation}	
	for any $\alpha \geq 1$. In particular, when $m'=1$, then
	\begin{equation}\label{5.3}
	\left[ \begin{matrix}
	1&0\\
	0&m q
	\end{matrix}\right]_{{\gamma}_q} \cong 
	\left[ \begin{matrix}
	m&1\\
	0&q
	\end{matrix}\right]_{{\gamma}_0}
	\end{equation}	
	for $m\geq 1$. Similarly, the above results for $\left[ \begin{matrix}
	\hat{M}&\hat{L}\\
	0&\hat{K}
	\end{matrix}\right]_{\hat{\gamma}_q}$ and $\left[ \begin{matrix}
	\hat{m}&\hat{\ell}\\
	0&\hat{k}
	\end{matrix}\right]_{\hat{\gamma_0}}$ can also be obtained.
\end{theorem}

\begin{proof}
	The proof is easily from Theorem \ref{theorem1-2}.
\end{proof}

In the following, we will use a sequence of skew-coordinates systems $\gamma_q=\left[ \begin{matrix}1&q\\
0&1
\end{matrix}\right], q\geq 1$, to compute the periodic entropy. By (\ref{5.3}), we have 
\begin{equation}\label{5.3.1}
\left[ \begin{matrix}
1&0\\
0&m q
\end{matrix}\right]_{{\gamma}_q} \cong 
\left[ \begin{matrix}
m&1\\
0&q
\end{matrix}\right]_{{\gamma}_0},
\end{equation}
for $m\geq 1$. The relation of (\ref{5.3.1}) enable us to compute the $1$-shift entropy $h_1(\mathcal{B})$ by computing the $\rho({\bf T}_{\gamma_q,1})$. We draw the pictures to illustrate ${\bf T}_{\gamma_q,1}$ for $q=1,2$ and $3$. The others can also be obtained inductively.
	\begin{align}
		{\bf T}_{\gamma_1,1}&=\begin{matrix}
			\includegraphics[scale=0.3]{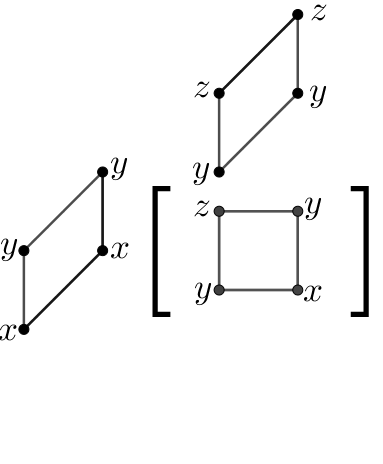}
		\end{matrix}\nonumber\\	
		&=\left[
		\begin{matrix}
			\includegraphics[scale=0.3]{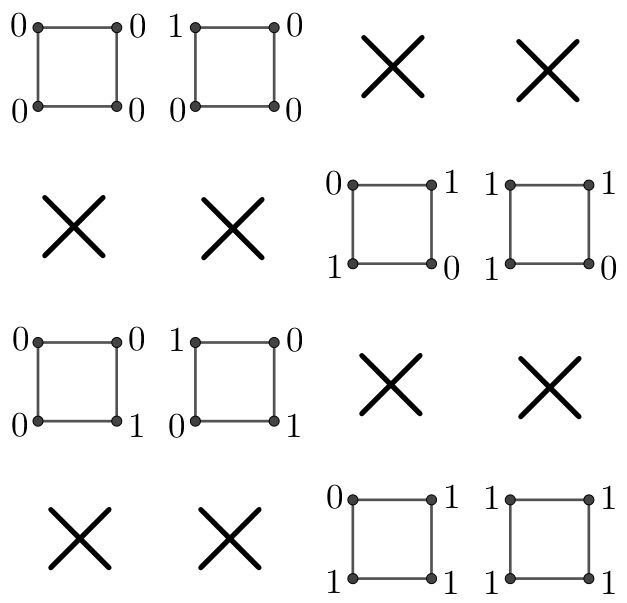}	
		\end{matrix}\right]\\
	    &=\left[ \begin{matrix}
	    x_{11}&x_{21}&\times &\times\\
	    \times & \times & x_{32}& x_{42}\\
	    x_{13}&x_{23}&\times &\times\\
	    \times & \times & x_{43}& x_{44}
	    \end{matrix} 
	    \right]_{2^2\times 2^2}.
	\end{align}
And
	\begin{align}
{\bf T}_{\gamma_2,1}&=\begin{matrix}
	\includegraphics[scale=0.3]{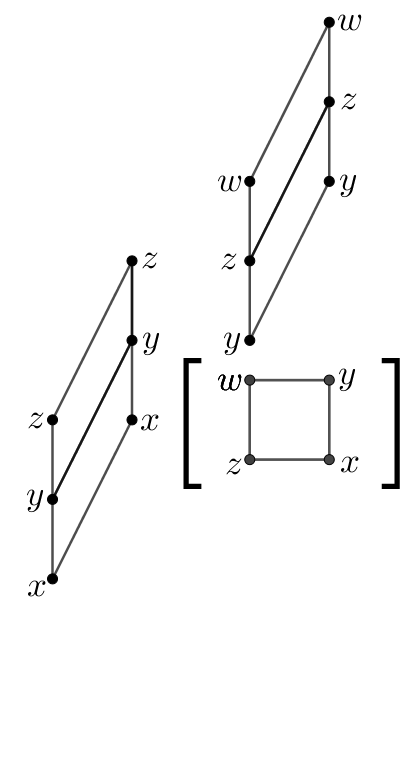}
\end{matrix}\nonumber\\	
&=\left[
\begin{matrix}
	\includegraphics[scale=0.35]{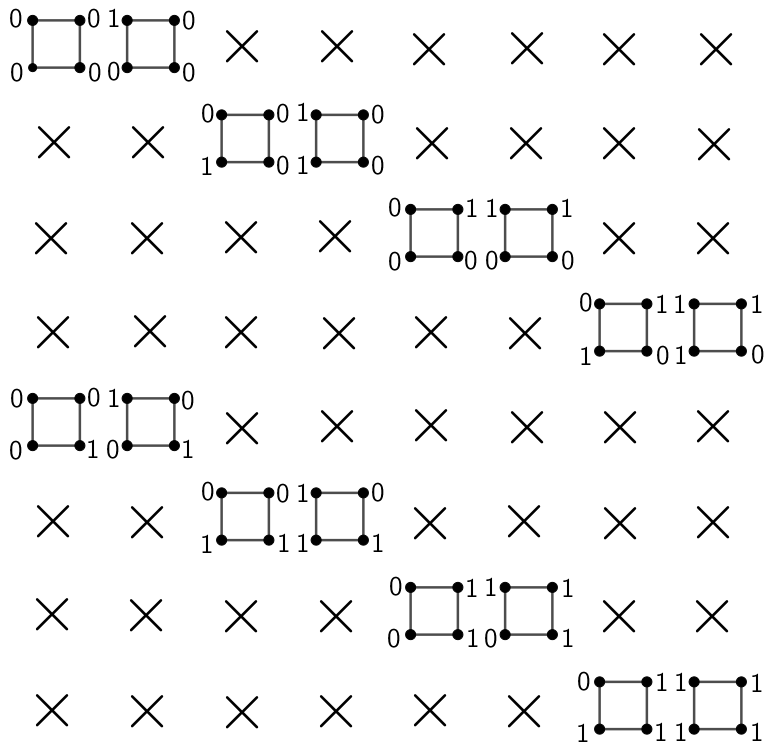}	
\end{matrix}\right]\\
&=\left[ \begin{matrix}
x_{11}&x_{21}&\times &\times &\times &\times &\times &\times \\
\times & \times & x_{31} & x_{41} &\times &\times &\times &\times\\
\times &\times &\times &\times& x_{12}&x_{22}&\times &\times\\
\times &\times &\times &\times&\times &\times& x_{32}& x_{42}\\
x_{13} & x_{23} & \times & \times & \times & \times & \times & \times \\ 
\times&\times &x_{33}&x_{43}&\times&\times&\times&\times \\
\times&\times&\times&\times &x_{14}&x_{24}&\times&\times \\
\times&\times &\times&\times&\times&\times&x_{34}&x_{44} 
\end{matrix} 
\right]_{2^3\times 2^3}\label{2.48}.
\end{align}
For $q\geq 3$,
\begin{equation}\label{2.50}
		{\bf T}_{\gamma_q,1}=\left[ \begin{matrix}
	 I_{2^{q-2}}	\otimes\left[ \begin{matrix}
			x_{11}&x_{21}&\times&\times\\
            \times&\times&x_{31} & x_{41}
		\end{matrix} 
		\right]& \\
		&I_{2^{q-2}}	\otimes\left[ \begin{matrix}
			x_{12}&x_{22}&\times&\times\\
			\times&\times&x_{32}& x_{42}
		\end{matrix} 
		\right] \\
		I_{2^{q-2}}	\otimes\left[ \begin{matrix}
			x_{13} & x_{23}&\times&\times\\
			\times&\times&x_{33}&x_{43}
		\end{matrix} 
		\right]& \\
		&I_{2^{q-2}}	\otimes\left[ \begin{matrix}
			x_{14}&x_{24}&\times&\times\\
			\times&\times&x_{34}&x_{44}
		\end{matrix} \right] 
		\end{matrix} \right]_{2^{q+1}\times 2^{q+1}}.
\end{equation}
Note that in (\ref{2.48}) to (\ref{2.50}), the $j$th column vector $(x_{1j},x_{2j},x_{3j},x_{4j})^t$ in ${\bf X}_{2\times2}$ has been arranged as 
\begin{equation*}
A_j=\left[ \begin{matrix}
x_{1j}&x_{2j}&\times&\times\\
\times&\times&x_{3j}&x_{4j}
  \end{matrix} \right].
\end{equation*}

For $r\geq 3$ and $\gamma\in GL_2(\mathbb{Z})$, ${\bf T}_{\gamma,n}$ can also be introduced. For example, for $r=3$,
\begin{equation*}
{\bf X}_{2\times 2}= \left[  x_{i,j} \right]_{3^2 \times 3^2}, 
\end{equation*}
\begin{equation}
{\bf T}_{\gamma_1,1}=\left[ \begin{matrix}
x_{11}&x_{21}&x_{31}&\times&\times&\times&\times&\times&\times\\
\times&\times&\times&x_{42}&x_{52}&x_{62}&\times&\times&\times\\
\times&\times&\times&\times&\times&\times&x_{73}&x_{83}&x_{93}\\
x_{14}&x_{24}&x_{34}&\times&\times&\times&\times&\times&\times\\
\times&\times&\times&x_{45}&x_{55}&x_{65}&\times&\times&\times\\
\times&\times&\times&\times&\times&\times&x_{76}&x_{86}&x_{96}\\
x_{17}&x_{27}&x_{37}&\times&\times&\times&\times&\times&\times\\
\times&\times&\times&x_{48}&x_{58}&x_{68}&\times&\times&\times\\
\times&\times&\times&\times&\times&\times&x_{79}&x_{89}&x_{99}
\end{matrix}    \right]_{3^2\times 3^2},
\end{equation}
In fact, for $q\geq 2$
\begin{equation}
		{\bf T}_{\gamma_q,1}=\left[ \begin{matrix}
			I_{3^{q-2}}	\otimes A_1& &\\
			&I_{3^{q-2}}	\otimes A_2& \\
			&&I_{3^{q-2}}	\otimes A_3 \\
			I_{3^{q-2}}	\otimes A_4& &\\
			&I_{3^{q-2}}	\otimes A_5& \\
			&&I_{3^{q-2}}	\otimes A_6\\
			I_{3^{q-2}}	\otimes A_7& &\\
			&I_{3^{q-2}}	\otimes A_8 & \\
			&&I_{3^{q-2}}	\otimes A_9
		\end{matrix} \right]_{3^{q+1}\times 3^{q+1}},
\end{equation}
where
\begin{equation*}
A_i=\left[ \begin{matrix}
	x_{1i}&x_{2i}&x_{3i}&\times&\times&\times&\times&\times&\times\\
	\times&\times&\times&x_{4i}&x_{5i}&x_{6i}&\times&\times&\times\\
	\times&\times&\times&\times&\times&\times&x_{7i}&x_{8i}&x_{9i}
\end{matrix} 
\right]
\end{equation*}
is a rearrangement of $i$th column vector in ${\bf X}_{2\times 2}$.

Since 
\begin{equation*}\label{5.4}
h_1(\mathcal{B})= \limsup_{(m,q)\rightarrow\infty}  \frac{1}{mq}\log \Gamma \left( \left[ \begin{matrix}
m&1\\
0&q
\end{matrix}\right]_{{\gamma}_0} \right)
\end{equation*}
and 
\begin{equation*}
\Gamma \left( \left[ \begin{matrix}
	1&0\\
	0&mq
\end{matrix}\right]_{{\gamma}_q} \right)=tr\left( {\bf T}^{mq}_{\gamma_q,1} \right),
\end{equation*}
we have
\begin{equation}\label{5.4-1}
h_1(\mathcal{B})= \limsup_{(m,q)\rightarrow\infty}  \frac{1}{mq}\log tr\left( {\bf T}^{mq}_{\gamma_q,1} \right).
\end{equation}

We need the following lemma to study (\ref{5.4-1}).

\begin{lemma}
	Let $A=[a_{i,j}]_{N\times N}$ be a non-negative matrix with eigenvalues $\{ \lambda_j \}_{j=1}^N$ and the maximum eigenvalue $\rho$. Then 
	\begin{equation}\label{5.7}
	tr(A^k)=\sum_{j=1}^N \lambda_j^k \leq N \rho^k,
	\end{equation}
	for all $k\geq 1$.
\end{lemma}

\begin{proof}
	Let $\{ \lambda_j \}_{j=1}^N$ be the eigenvalues of $A$. Then 
	\begin{equation*}
	\left| \lambda_j  \right| \leq \rho,
	\end{equation*}
	for all $j$.
	
	Hence 
	\begin{equation*}
	tr(A^k)=\sum_{j=1}^{N} \lambda_j^k \leq \sum_{j=1}^N \left| \lambda_j \right|^k \leq N\rho^k.
	\end{equation*}
Then, (\ref{5.7}) follows.
\end{proof}
Then we have the following result.

\begin{theorem}\label{thm 5.3}
	For any basic set $\mathcal{B}$ of admissible patterns, 
	\begin{equation}\label{5.9}
	h_1(\mathcal{B})=\limsup_{q \rightarrow \infty} \log\rho({\bf T}_{\gamma_q,1}).
	\end{equation}
\end{theorem}

\begin{proof}
	We first prove 
	\begin{equation}\label{5.10}
	h_1(\mathcal{B})\geq \limsup_{q \rightarrow \infty} \log\rho({\bf T}_{\gamma_q,1}).
	\end{equation}
	From (\ref{5.7}), we have
	\begin{equation*}
	\begin{aligned}
	\limsup_{m\rightarrow\infty} \frac{1}{mq}\log tr ({\bf T}_{\gamma_q,1}^{mq})&=	\limsup_{m\rightarrow\infty}  \frac{1}{mq}\log \rho_q^{mq} \left( 1+ \sum_{j=2}^{N_q} \left(\frac{\lambda_j}{\rho_q} \right)^{mq} \right)  \\
	&=\log\rho({\bf T}_{\gamma_q,1}) + 	\limsup_{m\rightarrow\infty}  \frac{1}{mq}\log  \left( 1+ \sum_{j=2}^{N_q} \left(\frac{\lambda_j}{\rho_q} \right)^{mq} \right)\\
	&=  \log\rho({\bf T}_{\gamma_q,1}),
	\end{aligned}
	\end{equation*}
where $\rho_q=\rho({\bf T}_{\gamma_q,1})$ and $N_q=r^{q+1}$, here $r$ is the number of symbols.

Then (\ref{5.10}) follows from
\begin{equation*}
h_1(\mathcal{B})\geq \limsup_{q\rightarrow\infty} \left(\limsup_{m\rightarrow\infty} \frac{1}{mq}\log tr ({\bf T}_{\gamma_q,1}^{mq})  \right) =\limsup_{q\rightarrow\infty}\log\rho({\bf T}_{\gamma_q,1}).
\end{equation*}
On the other hand, (\ref{5.7}) also implies 
\begin{equation*}
tr({\bf T}_{\gamma_q,1}^{mq})\leq N_q \rho({\bf T}_{\gamma_q,1})^{mq}.
\end{equation*}
Hence (\ref{5.9}) follows from
\begin{equation*}
h_1(\mathcal{B})\leq \limsup_{(m,q)\rightarrow\infty} \frac{1}{mq}\log N_q \rho({\bf T}_{\gamma_q,1})^{mq}= \limsup_{q\rightarrow\infty} \log \rho({\bf T}_{\gamma_q,1}).
\end{equation*}
\end{proof}
Combining Theorems \ref{thm 4.9}, \ref{thm 5.5} and \ref{thm 5.3}, we have
\begin{theorem}\label{thm 4.18}
If $\{{\bf T}_m\}_{m=1}^\infty$ is irreducible and uniformly dominated by $\{\rho({\bf T}_m)\}_{m=1}^\infty$, or if $\{{\bf T}_m\}_{m=1}^\infty$ is reducible and the maximum irreducible submatrices $\{\overline{\bf T}_m\}_{m=1}^\infty$ with indices pairs $i$ and $\sigma^{-1}(i)\in \overline{\mathcal{I}}_m$ and $\{\overline{\bf T}_m\}_{m=1}^\infty$ is uniformly dominated by $\{\rho({\bf T}_m)\}_{m=1}^\infty$. Then
	\begin{equation*}
h(\mathcal{B})=\limsup_{q \rightarrow \infty} \log\rho({\bf T}_{\gamma_q,1}).
\end{equation*}
\end{theorem}
From Theorem \ref{thm 4.18}, $h(\mathcal{B})$ can be studied by $\rho({\bf T}_{\gamma_q,1})$ which are easier than computing $\rho\left({\bf T}_m\right)$. 

\bibliographystyle{amsplain}

\end{document}